\documentclass{article}
\usepackage{float}
\usepackage{cite}
\usepackage{shapepar}
\usepackage{listings}
\usepackage{amsmath}
\usepackage{amsthm}
\usepackage{mathtools}
\usepackage{graphics}
\usepackage{amssymb}
\usepackage{makecell}
\usepackage{mathrsfs}
\usepackage{cite}
\usepackage{framed}
\usepackage{diagbox}
\usepackage{booktabs}
\usepackage{fancybox}
\usepackage{geometry}
\usepackage{multirow}
\usepackage{enumerate}
\usepackage{subcaption}
\usepackage{caption}
\usepackage{subcaption}
\usepackage{hyperref}
\usepackage{boldline}
\usepackage[table]{xcolor}
\usepackage{slashbox}
\usepackage{tabularray}
\usepackage{vcell}
\usepackage{titlesec}
\usepackage{graphicx}
\numberwithin{equation}{section}

\titleformat{\paragraph}
{\normalfont\normalsize\bfseries}{\theparagraph}{1em}{}
\titlespacing*{\paragraph}
{0pt}{3.25ex plus 1ex minus .2ex}{1.5ex plus .2ex}

\usepackage[ruled,linesnumbered]{algorithm2e}
\newtheorem{lemma}{Lemma}[section]
\newtheorem{remark}{Remark}

\SetCommentSty{mycommfont}

\newcommand\keywords[1]{\textbf{Keywords}: #1}

\SetCommentSty{mycommfont}

\title{Constructing probing functions for direct sampling methods\\ for inverse scattering problems with limited-aperture data:\\ finite space framework and deep probing network}

\author{Jianfeng Ning \thanks{School of Mathematics and Statistics, Wuhan University, Wuhan, China. ({ningjf@whu.edu.cn}).} \and Jun Zou \thanks{Department of Mathematics, The Chinese University of Hong Kong, Shatin, N.T., Hong Kong. The work of this author was substantially supported by the Hong Kong RGC General Research Fund (projects 14306623,  14308322 and 14306921). ({zou@math.cuhk.edu.hk}).}}
\date{}

\begin{document}
	\maketitle
 \begin{abstract}
This work studies an inverse scattering problem when limited-aperture data are available that are 
from just one or a few incident fields.
This inverse problem is highly ill-posed due to the limited receivers and a few incident fields employed. Solving inverse scattering problems with limited-aperture data is important in applications as collecting full data is 
often either unrealistic or too expensive. The direct sampling methods (DSMs) with full-aperture data can effectively and stably estimate the locations and geometric shapes of the unknown scatterers with a very limited number of incident waves. However, 
a direct application of DSMs to the case of limited receivers would face the resolution limit. To break this limitation, we propose a finite space framework with two specific 
schemes, and an unsupervised deep learning strategy to construct effective probing functions for the DSMs in the case with 
limited-aperture data. Several representative numerical experiments are carried out to illustrate and compare the performance of different proposed
schemes.
 \end{abstract}
\section{Introduction}

Inverse scattering problems have important applications in diverse areas such as geophysical exploration, medical imaging, radar and sonar imaging\cite{bulyshev2004three,persico2014introduction}.  These applications 
mainly involve two stages.  The first stage is data acquisition, where the unknown scatterers are probed by some incident waves, and the generated scattered fields are measured by an array of receivers.  The second stage is data processing, where numerical techniques are applied 
to process the measured data to recover the physical properties and geometries of unknown scatterers.  Both stages play a significant role in the accurate and stable reconstruction of unknown scatterers.  In the first stage, it is desired to have sufficient accurate measured data 
in order to get enough useful hidden information about the unknown scatterers, while in many real applications, such as in the underground mineral prospection, only very limited incident waves and receivers are available. In the second stage, the employed numerical schemes are desired to fully make use of the measured data to provide stable, accurate, and fast reconstruction. However, the methods such as optimization or iterative type \cite{chen2009subspace,bao2015inverse,borges2017high,van2021forward,langer2010investigation,audibert2022accelerated,ito2022least} can provide satisfactory numerical accuracies but might require unreasonable computational efforts. 
Non-iterative methods such as\cite{chen2013reverse,chen2013reverse2,ito2012direct,kirsch2002music,cakoni2011linear,harris2020orthogonality} can provide fast reconstruction but may sacrifice numerical accuracies in their reconstructions. 
In this work, we consider very important physical scenarios where only one or a few incident waves and 
very limited receivers are available, and propose some efficient schemes that are fast but 
still with acceptable numerical accuracies. 
Without full-aperture measurements, the ill-posedness of inverse scattering problems 
become much more severe.

For inverse problems with limited-aperture data available, the information from the ``shadow region" is very weak, which poses an extreme challenge to the recovery along the ``shadow region"\cite{li2015recovering,mager1976approach}. Various numerical methods\cite{li2015recovering,ahn2014study,bao2003numerical,mager1976approach,zinn1989optimisation,ahn2020fast,audibert2017generalized} have been proposed to directly process the limited-aperture data to recover the scatterer. By analyticity, the full-aperture data can be uniquely determined by the limited-aperture data, which serves as the motivation of another standard methodology that is to first recover the full data from limited-aperture data, then methods requiring full data can be applied\cite{dou2022data,liu2019data}. However, the analytic continuation is severely ill-posed, and thus, it is extremely difficult to recover full-aperture data accurately and stably. We also refer to \cite{fu2008simple,fu2009modified,cheng1998unique,lu2012unique,aussal2020data,ying2022analytic} for some stability analysis and numerical methods on analytic continuation.

In\cite{ito2012direct,ito2013direct,li2013direct,chow2022direct,ning2024direct}, 
some direct sampling methods (DSMs) are proposed 
for inverse acoustic and electromagnetic scattering problems
to provide a reasonable approximation for the shapes and locations of the scatterers with only one or a few incident fields. The DSMs have several distinct features, including robustness to noise, data efficiency, and very cheap computation. The fundamental component of DSMs is the construction of a set of special probing functions such that the resulting index function formed by the probing function and the measured data can take large values at points near the scatterers and relatively small values at points away from the scatterers. By further exploiting this idea, the DSMs have been extended to several other important inverse problems\cite{chow2014direct,chow2015direct,chow2021direct,chow2021directRadon}. However, although these DSMs can apply to the cases with only a few incident fields, the analysis and the construction of the 
required probing functions are all based on the case of full measurement receivers. For the case of limited-aperture measurement, a natural question is whether the probing functions of the existing DSMs can be directly applicable to limited-aperture data to provide a reasonable approximation of the locations and shapes of the scatterers. Another question is whether we can construct more suitable probing functions than those used in the full measurement case. In this work, we shall address these challenging technical issues 
for the inverse acoustic scattering problems when only limited-aperture data from just a few incidences are available.
Our first contribution is to show that the probing functions used for full-aperture data can also work stably for limited-aperture data, while the reconstruction resolution is low, especially in the ``shadow region". To break this limitation, our second contribution is to propose a general finite space framework and two specific schemes based on this framework to construct the probing functions for the cases with limited-aperture data available over a general partial area. In addition, we will also employ an unsupervised deep learning technique to construct the probing functions for limited-aperture measurement, which proposes the applications of deep learning to develop DSMs for inverse scattering problems 
in a brand-new aspect.

In recent years, deep learning techniques have been widely exploited for solving various inverse problems. The well-known physics-informed neural network (PINN)\cite{raissi2019physics} was proposed for solving forward and inverse problems in PDEs, where the solutions are constructed by neural networks and are updated to satisfy the strong PDE equations. In \cite{li2020nett}, the authors proposed a method called NETT that employs neural networks as regularization functions. In\cite{adler2017solving}, the authors proposed a gradient-like iterative scheme to learn the gradient component by using neural networks. In a recent survey\cite{tanyu2022deep}, the authors studied and compared the applications of several well-known neural operators, e.g. 
\cite{lu2021learning,li2020fourier}, 
for solving inverse problems 
by employing the Tikhonov regularization combined with a trained forward neural operator instead of classical forward-problems solvers. We refer interested readers to \cite{arridge2019solving} for a review on data-driven based deep learning methods for inverse problems.

Several deep learning methods, especially data-driven methods, have also been applied to solve inverse scattering problems. Methods employing well-designed neural networks with measured data as input and approximation to the true scatterer as the output can be found in\cite{khoo2019switchnet,gao2022artificial,hu2024residual}. Two-step methods, e.g.\cite{wei2018deep,ning2023direct,ning2024direct}, consist of an initial guess from a fast classical method and a refinement step through neural networks. Contrastly, in \cite{zhou2023neural}, an initial guess of the scatterers is first obtained from the deep learning method and then refined by the recursive linearization method. In \cite{li2024reconstruction}, a learned projected algorithm was proposed to learn the projector and a prior information of the unknown scatterers. By employing the latent representation of surfaces with neural networks, an iterative algorithm to solve an inverse obstacle scattering problem was proposed in\cite{chen2023solving}. We refer to \cite{chen2020review} for an overview of deep learning methods for inverse scattering problems. 

We shall also propose a deep learning approach for constructing the probing functions, 
and the scheme has several promising advantages. Firstly, it is very different from the data-driven type methods. 
By designing a novel loss function, the training process of 
the method does not require any data as well as the information of the incident fields, 
except the wavenumber that can also be relaxed (see the discussions in Section \ref{sec:deeplearning}).
Secondly, once the neural network is trained, the resulting DSM can be computed very cheaply for a given example. 
We do not need to retrain the network for a new instance with the same measurement configuration, 
which is also very different from PINN. In addition, unlike our proposed methods based on the finite space framework, 
the deep learning scheme does not need to choose the regularization parameters either.

The rest of this paper is organized as follows. We present in Section \ref{sec:probForm} the inverse scattering problems 
considered in this work, and review in Section \ref{sec:DSMfull} the DSMs for full-aperture data and provide some analysis 
on their direct extension to inverse scattering problems with limited-aperture data. 
The main contributions of our work are addressed in Sections \ref{sec:FEM} and \ref{sec:deeplearning}, 
where we construct the probing functions for limited-aperture data by a finite space framework with 
two specific schemes given, as well as a deep learning technique. 
Several numerical experiments are carried out in Section \ref{sec:numeri} to check and compare different schemes 
proposed in this work, and some concluding remarks are made in Section \ref{sec:conclu}.

\section{Problem formulations}
\label{sec:probForm}
We shall focus on the inverse scattering problems that aim to recover the unknown scatterers, 
when only limited-aperture data are measured from just one or a few incident fields. Let $G(x,y)$ be the free-space Green's function of the scattering problem given by \cite{colton2019inverse}
\begin{equation}
	G(x,y)=\left\{
	\begin{aligned}
		\frac{\text{i}}{4}H_{0}^{(1)}(k|x-y|) \quad \text{in}\quad \mathbb{R}^2,\\
		\frac{\exp(\text{i}k|x-y|)}{4\pi|x-y|} \quad \text{in}\quad \mathbb{R}^3.
	\end{aligned}
	\right.
\end{equation}
The function $H_{0}^{(1)}$ refers to the zeroth-order Hankel function of the first kind and $k$ is the wavenumber. Suppose that in the homogeneous background space a
bounded domain $D$ is occupied by some inhomogeneous media. With the incident plane wave $u^{i}(x,d)=e^{\text{i}kx\cdot d}, d\in \mathbb{S}^{N-1}$($N=2$ or $3$), the total field $u=u^i+u^s$ induced by the scatterers satisfies the following Helmholtz equation\cite{colton2019inverse}:
\begin{equation}
	\Delta u + k^2n(x)u=0 \quad \text{in} \quad\mathbb{R}^N, \label{equa:helmo}
\end{equation}
\begin{equation}
	\lim_{r\rightarrow\infty}r^{(N-1)/2}(\frac{\partial u^s}{\partial r}-\text{i}ku^s)=0,
	\end{equation}
where $r=|x|$, and $n(x)$ is the refractive index that is equal to 1 inside the homogeneous background medium.
For impenetrable obstacles, equation \eqref{equa:helmo} is replaced by
\begin{equation}
	\Delta u + k^2u=0 \quad \text{in} \quad \mathbb{R}^N\setminus D,
\end{equation} 
with a boundary condition on $\partial D$ depending on the nature of the scatterers. It is known that the scattered field $u^s=u-u^{i}$ satisfies the following asymptotic behavior\cite{colton2019inverse}:
\begin{equation}
	u^{s}(x) = \frac{\exp(\text{i}k|x|)}{|x|^{(N-1)/2}}\Big\{u^{\infty}(\hat{x}) + \mathcal{O}(1/|x|)\Big\}, \quad |x|\rightarrow\infty,
	\label{Asym1}
\end{equation}
which holds uniformly for all $\hat{x}=x/|x|\in \mathbb{S}^{N-1}$, and $u^{\infty}$ is called the far-field pattern of $u^{s}$. The far-field pattern of the Green's function is given by
\begin{equation}
	G^{\infty}(z,\hat{x}_r)=\left\{
	\begin{aligned}
		&\frac{\exp(\text{i}\pi/4)}{\sqrt{8k\pi}}\exp(-\text{i}k\hat{x}_r\cdot z)  \quad &\text{in}\quad \mathbb{R}^2,\\
		&\frac{1}{4\pi}\exp(-\text{i}k\hat{x}_r\cdot z) \quad &\text{in}\quad \mathbb{R}^3.
	\end{aligned}\right.
 \label{equa:farfieldG}
\end{equation}
The scattered near-field $u^s$ and far-field pattern $u^\infty$ also have
the following important expressions:
\begin{equation}
\begin{split}
    u^s(x)=\int_D G(y,x)I(y)dy, \quad 
    u^\infty(\hat{x})=\int_D G^\infty(y,\hat{x})I(y)dy, 
\end{split}
 \label{equa:integralForm}
\end{equation}
where $I(y)=(n(y)-1)k^2u(y)$ is called the induced current.

The inverse scattering problem we are interested in this work is to recover the support $D$ of the scatterers 
from some limited-aperture far-field data corresponding to one or a few incident waves. Without loss of generosity, we give the detailed presentation and analysis only for the 2-D case, 
while all the results would hold also for the 3D case with only some natural modifications.
We assume that the limited-aperture measurement data $u^\infty(x)$ is measured over $L$ disjoint  
curves on $\mathbb{S}^1$, $\Gamma = \cup_{l=1}^{L} \Gamma_l$, with
\begin{equation}
	\Gamma_l = \{(\cos\theta,\sin\theta)| \theta\in(-\alpha_l+\beta_l,\alpha_l+\beta_l)\}, \quad \alpha_l\in(0,\pi), \beta_l \in(-\pi, \pi),
\end{equation}
and $\Gamma_l \cap \Gamma_q = \emptyset$ for $l\ne q$. We can see that each separated curve 
$\Gamma_l$ has an angle of width $2\alpha_l$ and centers at the angle $\beta_l$. 
Note that in this work, ``full-aperture data" means that we have a full set of receivers on $\mathbb{S}^1$ to collect 
the data, but from just one or a few incident fields. 
``Limited-aperture data" means that the receivers are incomplete, but also from just one or a few incidences.

\section{ DSM with full-aperture data and its direct extension to limited-aperture data}
\label{sec:DSMfull}
In this section, we review the DSM with full-aperture data and conduct some analysis on its direct extension to limited-aperture data. For far-field data $u^\infty$ measured from the complete area $\mathbb{S}^1$, the DSM employs $G^\infty(z,\hat{x})$ as the probing function and the index function in a sampling domain $\Omega$ is computed as \cite{li2013direct}
\begin{equation}
	\mathcal{I}_{\text{full}}(z)=\big|\langle G^\infty(z,\hat{x}),u^\infty(\hat{x})\rangle_{\mathbb{S}^1}\big|=\bigg|\int_{\mathbb{S}^1} G^\infty(z,\hat{x})\overline{u^\infty(\hat{x}))}d\hat{x}\bigg|, \quad z\in \Omega.
 \label{equ:DSMfull}
\end{equation}

We know \cite{li2013direct} that 
if the index function $\mathcal{I}_{\text{full}}(z)$ takes a large value at $z$, the sampling point $z$ is likely to be within or near the scatterers. If it takes a relatively small value, the sampling point $z$ is likely to be away from the scatterers. Thus, the index function provides an approximation of the locations and shapes of the unknown scatterers. For limited-aperture data $u^\infty$ 
measured from $\Gamma$, a direct choice for the probing function is $G^\infty(z,\hat{x})$, i.e., like 
the full-aperture data. The corresponding index function is then computed as
\begin{equation}
	\mathcal{I}_\Gamma(z)=\big|\langle G^\infty(z,\hat{x}),u^\infty(\hat{x})\rangle_{\Gamma}\big|, \quad z\in \Omega.
	\label{equ:DDSMpartial}
\end{equation}

We have the following stability analysis for the index functions \eqref{equ:DSMfull} and \eqref{equ:DDSMpartial}. 
\begin{lemma}
    Let $u^\infty$ be the exact data and $u^\infty_\delta$ as the measured data containing noise. It holds that 
    \begin{equation}\label{equa:stabeanalys}
    \begin{split}
         |\mathcal{I}_{\text{full}}(z)-\mathcal{I}_{\text{full}}^\delta(z)|&\le \frac{1}{2\sqrt{k}}\Vert u^\infty(\hat{x})-u^\infty_\delta(\hat{x})\Vert_{L^2(\mathbb{S}^1)},\\
           \big|\mathcal{I}_{\Gamma}(z)-\mathcal{I}_{\Gamma}^\delta(z)\big|&\le \frac{|\Gamma|}{4\pi\sqrt{k}}\Vert u^\infty(\hat{x})-u^\infty_\delta(\hat{x})\Vert_{L^2(\Gamma)},
    \end{split}
    \end{equation}
    where $|\Gamma|$ denotes the measure of $\Gamma$.
\end{lemma}
\begin{proof}
We show only the first inequality, as the second inequality can be carried out similarly. By direct computation, we have
    \begin{equation*}
	\begin{split}	|\mathcal{I}_{\text{full}}(z)-\mathcal{I}_{\text{full}}^\delta(z)|&=\big| |\langle G^\infty(z,\hat{x}),u^\infty(\hat{x})\rangle_{\mathbb{S}^1}|-|\langle G^\infty(z,\hat{x}),u^\infty_\delta(\hat{x})\rangle_{\mathbb{S}^1}|\big|\\
	&\le\Vert G^\infty(z,\hat{x})\Vert_{L^2(\mathbb{S}^1)}\Vert u^\infty(\hat{x})-u^\infty_\delta(\hat{x})\Vert_{L^2(\mathbb{S}^1)}\\
	&=\frac{1}{2\sqrt{k}}\Vert u^\infty(\hat{x})-u^\infty_\delta(\hat{x})\Vert_{L^2(\mathbb{S}^1)}.
	\end{split}
\end{equation*}
\end{proof}

The above stability results show that the index function is very robust to noise. In addition, we see that 
$G^\infty(z,\hat{x})$ is very smooth and is dominated mainly by low-frequency modes if $k|z|$ is not too large. While the noise is 
primarily presented in high-frequency modes, it is likely to be smoothed out via the integration (the inner product), 
therefore the reconstruction by DSM can be very robust. 

For limited-aperture measured data, an important question is whether the index function computed via \eqref{equ:DDSMpartial} takes a large value for a sampling point $z$ near the scatterers and a relatively small value for point $z$ far away from the scatterers. 
To answer this question, we define a function:
\begin{equation}
		K_\Gamma(z,y):=\langle G^\infty(z,\hat{x}), G^\infty(y,\hat{x})\rangle_{\Gamma}, \quad y,z \in\Omega.
\end{equation}
By the integral representation\eqref{equa:integralForm} of $u^\infty$ and a general numerical quadrature rule, we can write 
\begin{equation}
	u^\infty(\hat{x})=\int_D G^\infty(y,\hat{x})I(y)dy\approx \sum_{j}w_j G^\infty(y_j,\hat{x}), \quad y_j \in D
 \label{equa:disIntegtal}
\end{equation}
for a set of quadrature points $\{y_j\} \subset D$, then we have 
\begin{equation}\label{eqn:rules}
	\mathcal{I}_\Gamma(z)=|\langle G^\infty(z,\hat{x}),u^\infty(\hat{x})\rangle_\Gamma|\approx \bigg|\sum_{j}w_j K_\Gamma(z,y_j)\bigg|, \quad y_j\in D.
\end{equation}

We see from \eqref{eqn:rules} that 
if $|K_\Gamma(z,y)|$ can take a relatively large value when $z$ is near the point $y$ 
and takes a relatively small value when $z$ is away from the point $y$, then by $\{y_j\}\subset D$, $\mathcal{I}_\Gamma(z)$ is likely to take a relatively large value if $z$ is within $D$ and take a small value 
if  $z$ is away from $D$. In this case $G^\infty(z,\hat{x})$ can be a reasonable choice of probing functions. 
To achieve a high-resolution reconstruction, the decay rate of $|K_\Gamma(z,y)|$ is desired to be fast when $|z-y|$ increases. For the full-aperture data case $\Gamma=\mathbb{S}^1$, we have $K_{\mathbb{S}^1}(z,y)=\frac{1}{4k}J_0(k|z-y|)$ \cite{li2013direct} with $J_0$ being the Bessel functions of order $0$, which clearly takes the maximum values when $z$ is near $y$ and decays as $|z-y|$ increases. 
In addition, the decay rate becomes faster as the wavenumber $k$ increases.

We now conduct some analysis on the behavior of $|K_\Gamma(z,y)|$ with limited-aperture measurement curve $\Gamma$. We first 
cite the following result from \cite{chen2013reverse}, which is needed for the subsequent analysis.
\begin{lemma}
	For any $-\infty <a<b<\infty$ and real-valued function $u\in C^2[a,b]$ thats satisfies $|u'(t)|\ge 1$ for $t\in (a,b)$. Assume that $a=x_0<x_1<\cdots<x_M=b$ is a division of $(a,b)$ such that $u'$ is monotone in each interval $(x_{i-1},x_i),i=1,\cdots,M$. Then for any function $\phi$ defined on $(a, b)$ with integrable derivative and for any $\lambda>0$, we have
	\begin{equation}
		\bigg|\int_{a}^{b}e^{\text{i}\lambda u(t)}\phi(t)dt\bigg|\le (2M+2)\lambda^{-1}\bigg[|\phi(b)|+\int_{a}^{b}|\phi'(t)|dt\bigg].
	\end{equation}
	\label{oscillatory}
\end{lemma}
We now show some behavior of $|K_\Gamma(z,y)|$ with limited-aperture measurement curve $\Gamma$. For the ease of 
exposition, we consider $\Gamma=\big\{(\cos\theta,\sin\theta)\big|\theta\in(-\alpha,\alpha)\big\}$ with $ \alpha\in(0,\pi]$, 
while the analysis can be extended to general cases easily.
\begin{lemma}
    Let $R=|y-z|$. Then for fixed $\alpha$, we have
    \begin{equation}
	|K_\Gamma(z,y)|\le Ck^{-3/2}R^{-1/2} \label{equa:esta2},
    \end{equation}
   where $C$ is a constant independent of $\alpha$ and $kR$, and
    \begin{equation}
	\lim_{kR\rightarrow 0}|K_\Gamma(z,y)|=\frac{\alpha}{4k\pi},
	\label{equa:lim1}
\end{equation}
which is the maximum value of $|K_\Gamma(z,y)|$. On the other hand, for fixed $kR$ and small $\alpha$, we have
\begin{equation}
	|K_\Gamma(z,y)|=\frac{\alpha}{4k\pi}+ o(\alpha).\label{equa: smallalpha}
\end{equation}

\begin{proof}
    Firstly, by denoting $y-z=R(\cos\beta,\sin\beta)$ and $\hat{x}$=($\cos\theta,\sin\theta$), we can observe that 
\begin{equation}
|K_\Gamma(z,y)|=\frac{1}{8k\pi}\bigg|\int_{-\alpha}^{\alpha}e^{\text{i}k(y-z)\cdot\hat{x}}d\theta\bigg|=\frac{\alpha}{8k\pi^2}\bigg|\int_{-\pi}^{\pi}e^{\text{i}kR\cos(\frac{\alpha}{\pi}\theta-\beta)}d\theta\bigg|.
\label{equa:39}
\end{equation}
Equations \eqref{equa:lim1} and \eqref{equa: smallalpha} can then be derived easily by directly 
taking the limit of \eqref{equa:39}.

We now prove $ \lim_{kR\rightarrow \infty}K_\Gamma(z,y)=0$. Denote $\delta=(kR)^{-1/2}$ and $Q_\delta=\{\theta\in (-\pi,\pi):|\frac{\alpha}{\pi}\theta-\beta-m\pi|<\delta,m=0,\pm1\}$.  Then we have
\begin{equation}\label{equa:313}
	\bigg|\int_{Q_\delta}e^{\text{i}kR\cos(\frac{\alpha}{\pi}\theta-\beta)}d\theta\bigg|\le 4\pi\alpha^{-1}(kR)^{-1/2} . 
\end{equation}
On the other hand, for $\theta\in (-\pi,\pi)\backslash  Q_\delta$ and small $\delta$, we can deduce 
\begin{equation}\label{equa:314}
	2\pi\alpha^{-1}(kR)^{1/2}\big|\cos\big(\frac{\alpha}{\pi}\theta-\beta\big)'\big|=
	2(kR)^{1/2}\big|\sin\big(\frac{\alpha}{\pi}\theta-\beta\big)\big|\ge 2\frac{\sin\delta}{\delta}\ge 1.
\end{equation}
Thus by Lemma \ref{oscillatory} with $u=2\pi\alpha^{-1}(kR)^{1/2}\cos(\frac{\alpha}{\pi}\theta-\beta),\lambda=(2\pi)^{-1}\alpha(kR)^{1/2}$ and $\phi=1$, we obtain
\begin{equation}\label{equa:315}
	\bigg|\int_{(-\pi,\pi)\backslash Q_\delta}e^{\text{i}kR\cos(\frac{\alpha}{\pi}\theta-\beta)}d\theta\bigg|\le C \alpha^{-1}(kR)^{-1/2},
\end{equation}
where $C$ is independent of $\alpha$ and $kR$.  Now the estimate \eqref{equa:esta2} follows readily 
from \eqref{equa:313} and \eqref{equa:315}.
\end{proof}
\end{lemma}

Equations \eqref{equa:esta2} and \eqref{equa:lim1} indicate that $|K_\Gamma(z,y)|$ takes a relatively 
large value when $z$ is close to $y$ while taking a small value when $z$ is away from $y$. 
This means that $G^\infty(z,\hat{x})$ can be a reasonable probing function for limited-aperture data, 
as we remarked earlier right after \eqref{eqn:rules}. 
However, equation \eqref{equa:lim1} together with equation \eqref{equa: smallalpha} indicates 
that for small $\alpha$, the decay rate of $K_\Gamma(z,y)$ is slow, 
and we may achieve only low-resolution reconstruction. 
Fig.\ref{fig: decayrate} shows the decay behavior of $|K_\Gamma(z,y)|$ with $k=8, \alpha=\pi/3$ and serveral different $\beta$, 
from which we observe that the decay rate is slower along the direction $\beta=0$ than other directions. 
This explains the big challenge for the reconstruction in the ``shadow region", 
and motivates a major focus of this work, i.e., to construct suitable probing functions 
when limited-aperture measurement data is available so that we may break 
the limit of low-resolution reconstruction to some extent.
\begin{figure}[htp]
		\centering
	\includegraphics[width=0.45\linewidth]{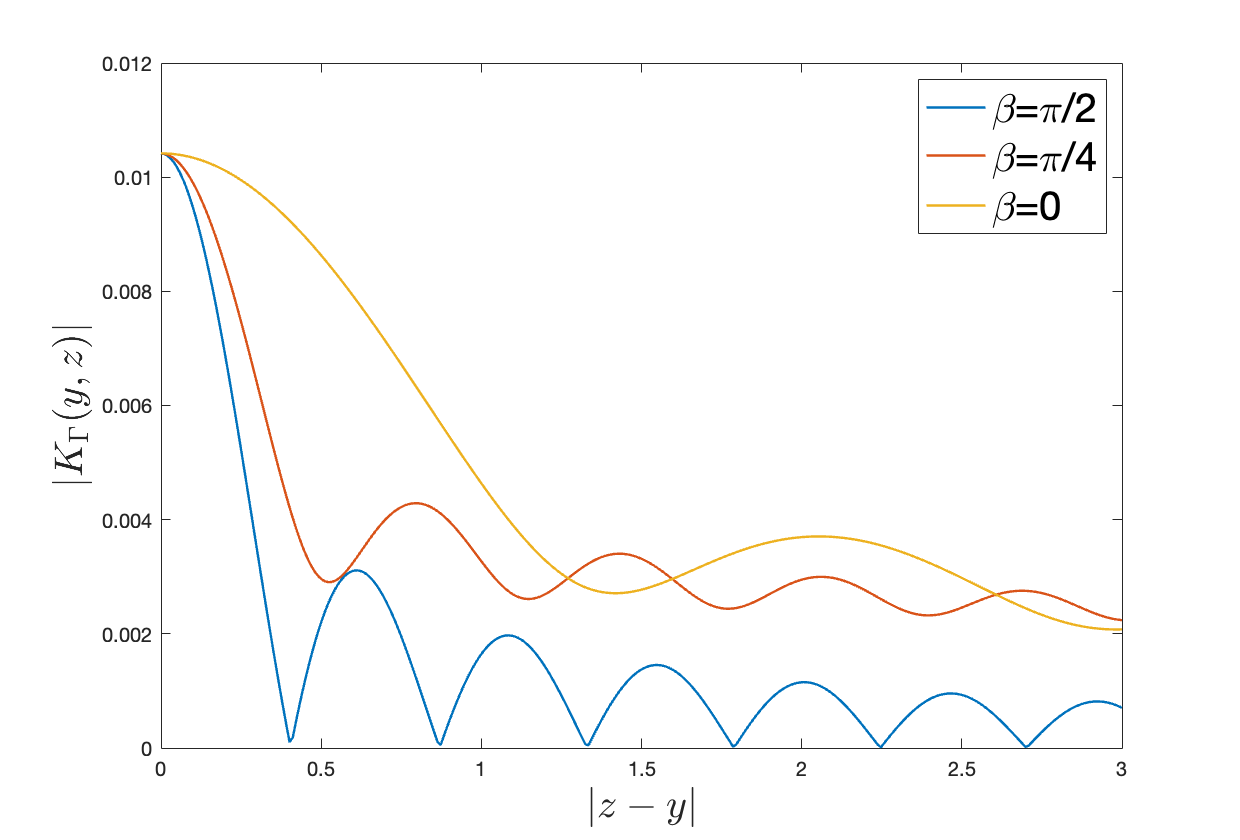}
	\caption{The decay behaviour of $|K_\Gamma(z,y)|$ with $k=8, \alpha=\pi/3$ and different $\beta$.}
	\label{fig: decayrate}
\end{figure}

\section{A finite space framework for constructing probing functions}
\label{sec:FEM}
In this section, we first propose a general finite space framework for constructing an effective probing functions 
for the DSM with limited-aperture data measured from the general partial area $\Gamma$. Then, based on this framework and the properties of the far-field data, 
we introduce two specific schemes to construct the probing functions. The idea of this framework is motivated by the unique continuation of the analytic function. Due to the unique continuation of the analytic function, the limited-aperture data $u^\infty|_\Gamma$ contains the whole information of the full-aperture data $u^\infty|_{\mathbb{S}^1}$, which motivates us to find a probing function $G_\Gamma(z,\hat{x})$ such that
\begin{equation}
	 	 \langle G_\Gamma(z,\hat{x}), u^\infty(\hat{x})\rangle_{L^2(\Gamma)}\approx\langle G^\infty(z,\hat{x}), u^\infty(\hat{x})\rangle_{L^2(\mathbb{S}^1)}, \quad u^\infty\in C^\infty(\mathbb{S}^1).
\end{equation}
To find the probing function satisfying the above property, we introduce two finite dimensional spaces:
\begin{equation}
	\mathcal{U}=\text{Span}\{\psi_1,\psi_2,\cdots,\psi_M\}\subset L^2(\Gamma),
\end{equation} 
\begin{equation}
	\mathcal{V}=\text{Span}\{\varphi_1,\varphi_2,\cdots,\varphi_N\}\subset C^\infty(\mathbb{S}^1).
\end{equation}
We call $\mathcal{U}$ the trial space and $\mathcal{V }$ the testing space. The trial space $\mathcal{U}$ consists of functions that are used to approximate the probing function $G_\Gamma(z,\hat{x})$, while the testing space $\mathcal{V }$  consists of functions we use to impose the constraints. Specifically, we construct the probing function $G_\Gamma(z,\hat{x})$ with a finite expansion $\sum_{m=1}^{M}f_m(z)\psi_m(\hat{x})$, where the coefficient functions $\{f_m(z)\}_{m=1}^{M}$ are to be determined such that
\begin{equation}
		\langle \sum_{m=1}^{M}f_m(z)\psi_m(\hat{x}),v(\hat{x})\rangle_{L^2(\Gamma)}\approx \langle G^\infty(z,\hat{x}), v(\hat{x})\rangle_{L^2(\mathbb{S}^1)}, \quad \forall v\in \mathcal{V}, \,z\in \Omega.
\end{equation}
By the linearity of the inner product, this is equivalent to 
\begin{equation}
	\sum_{m=1}^{M}f_m(z)\langle \psi_m(\hat{x}),\varphi_n(\hat{x})\rangle_{L^2(\Gamma)}\approx \langle G^\infty(z,\hat{x}), \varphi_n(\hat{x})\rangle_{L^2(\mathbb{S}^1)} \quad \text{for} \quad n=1,2,\cdots,N; \, z\in \Omega.
	\label{equa:FEMf1}
\end{equation}
The above equation then defines a matrix $\mathbb{A}\in \mathbb{C}^{N\times M }$ and a vector function $\mathbf{B}(z)\in  \mathbb{C}^N, z\in \Omega$ with entries
\begin{equation}
	\mathbb{A}_{nm}=\langle \psi_m(\hat{x}),\varphi_n(\hat{x})\rangle_{L^2(\Gamma)}, \quad \mathbf{B}_n(z)=\langle G^\infty(z,\hat{x}), \varphi_n(\hat{x})\rangle_{L^2(\mathbb{S}^1)},
 \label{equa:FEMAB}
\end{equation}
for  $m=1,2,\cdots, M$ and $n=1,2,\cdots,N$. 
Note that the matrix $\mathbb{A}$ depends on $\mathcal{U},\mathcal{V}$ and $\Gamma$, while the vector function $\mathbf{B}(z)$ depends on the testing space $\mathcal{V}$ and the sampling point $z$. By further writing $\mathbf{F}(z)=(f_1(z),f_2(z),\cdots,f_M(z))^{T}$, equation \eqref{equa:FEMf1} can be rewritten as
\begin{equation}
\mathbb{A}(\mathbf{F}(z))\approx \mathbf{B}(z), \quad z \in \Omega.
\label{equa:FEMf2}
\end{equation}
The system \eqref{equa:FEMf2} is likely to be ill-conditioned due to the severely ill-posed nature of limited-aperture data.
We apply a regularization strategy to solve the system for constructing a probing function $G_\Gamma(z,\hat{x})$, 
which is expected to provide a stable approximation. 
More precisely, we have the following abstract error estimate of the approximation for this general framework.

%

\begin{lemma}
    Let $\mathcal{P}u^\infty$ be the projection of $u^\infty$ into the space $\mathcal{V}$, $u^\infty_\delta$ 
    the measurement data, and $\varepsilon_1=\Vert u^\infty-\mathcal{P}u^\infty\Vert_{L^2(\mathbb{S}^1)}, \varepsilon_2=\Vert u^\infty-\mathcal{P}u^\infty\Vert_{L^2(\Gamma)}$ and $\delta=\Vert u^\infty-u^\infty_\delta\Vert_{L^2(\Gamma)}$. Then we have
\begin{equation}
    \begin{split}
           &\big|\langle G_\Gamma(z,\hat{x}), u^\infty_\delta(\hat{x})\rangle_{L^2(\Gamma)}-\langle G^\infty(z,\hat{x}), u^\infty(\hat{x})\rangle_{L^2(\mathbb{S}^1)}\big|
           \\
           \le &(\varepsilon_2+\delta)\Vert G_\Gamma(z,\cdot)\Vert_{L^2(\Gamma)}+\varepsilon_1\Vert G^\infty(z,\cdot)\Vert_{L^2(\mathbb{S}^1)}+\big|\langle G_\Gamma(z,\hat{x}),\mathcal{P}u^\infty(\hat{x})\rangle_{L^2(\Gamma)}-\langle G^\infty(z,\hat{x}), \mathcal{P}u^\infty(\hat{x})\rangle_{L^2(\mathbb{S}^1)}\big|,
    \end{split}   
    \label{equa:errorest}
    \end{equation}
    or  \begin{equation}
    \begin{split}
           &\big|\langle G_\Gamma(z,\hat{x}), u^\infty_\delta(\hat{x})\rangle_{L^2(\Gamma)}-\langle G^\infty(z,\hat{x}), u^\infty(\hat{x})\rangle_{L^2(\mathbb{S}^1)}\big|
           \\
           \le &(\varepsilon_2+\delta)\Vert G_\Gamma(z,\cdot)\Vert_{L^2(\Gamma)}+\varepsilon_1\Vert G^\infty(z,\cdot)\Vert_{L^2(\mathbb{S}^1)}+\eta\Vert \mathcal{P}u^\infty\Vert_{L^2(\mathbb{S}^1)},
    \end{split}   
    \label{equa:errorest2}
    \end{equation}
    where
     \begin{equation}
        \eta:=\sup_{v\in \mathcal{V}} \frac{\big|\langle G_\Gamma(z,\hat{x}),v\rangle_{L^2(\Gamma)}-\langle G^\infty(z,\hat{x}), v\rangle_{L^2(\mathbb{S}^1)}\big|}{\Vert v\Vert_{L^2(\mathbb{S}^1)}}.
    \end{equation}
        \label{lemma:41}
\end{lemma}
    \begin{proof}
    By introducing some intermediate terms and using the triangle and Cauchy–Schwarz inequalities, we derive immediately 
        \begin{equation}
	\begin{split}
		&\big|\langle G_\Gamma(z,\hat{x}), u^\infty_\delta(\hat{x})\rangle_{L^2(\Gamma)}-\langle G^\infty(z,\hat{x}), u^\infty(\hat{x})\rangle_{L^2(\mathbb{S}^1)}\big|\\
		\le&\big|\langle G_\Gamma(z,\hat{x}), u^\infty_\delta(\hat{x})-u^\infty(\hat{x})\rangle_{L^2(\Gamma)}\big|+\big|\langle G_\Gamma(z,\hat{x}), u^\infty(\hat{x})\rangle_{L^2(\Gamma)}-\langle G^\infty(z,\hat{x}), u^\infty(\hat{x})\rangle_{L^2(\mathbb{S}^1)}\big|\\
		\le& \Vert G_\Gamma(z,\cdot)\Vert_{L^2(\Gamma)}\Vert u^\infty_\delta-u^\infty\Vert_{L^2(\Gamma)}+\big|\langle G_\Gamma(z,\hat{x}),\mathcal{P}u^\infty(\hat{x})\rangle_{L^2(\Gamma)}-\langle G^\infty(z,\hat{x}), \mathcal{P}u^\infty(\hat{x})\rangle_{L^2(\mathbb{S}^1)}\big|\\
		&+\big|\langle G_\Gamma(z,\hat{x}), u^\infty(\hat{x})-\mathcal{P}u^\infty(\hat{x})\rangle_{L^2(\Gamma)}+\langle G^\infty(z,\hat{x}), \mathcal{P}u^\infty(\hat{x})-u^\infty(\hat{x})\rangle_{L^2(\mathbb{S}^1)}\big|\\
		\le  &\Vert G_\Gamma(z,\cdot)\Vert_{L^2(\Gamma)}\big(\Vert u^\infty_\delta-u^\infty\Vert_{L^2(\Gamma)}+\Vert u^\infty-\mathcal{P}u^\infty\Vert_{L^2(\Gamma)}\big)+\Vert G(z,\cdot)\Vert_{\mathbb{S}^1}\Vert u^\infty-\mathcal{P}u^\infty\Vert_{L^2(\mathbb{S}^1)}\\
		&+\big|\langle G_\Gamma(z,\hat{x}),\mathcal{P}u^\infty(\hat{x})\rangle_{L^2(\Gamma)}-\langle G^\infty(z,\hat{x}), \mathcal{P}u^\infty(\hat{x})\rangle_{L^2(\mathbb{S}^1)}\big|\\
       =&(\varepsilon_2+\delta)\Vert G_\Gamma(z,\cdot)\Vert_{L^2(\Gamma)}+\varepsilon_1\Vert G^\infty(z,\cdot)\Vert_{L^2(\mathbb{S}^1)}+\big|\langle G_\Gamma(z,\hat{x}),\mathcal{P}u^\infty(\hat{x})\rangle_{L^2(\Gamma)}-\langle G^\infty(z,\hat{x}), \mathcal{P}u^\infty(\hat{x})\rangle_{L^2(\mathbb{S}^1)}\big|.
	\end{split}
\end{equation}
Then the inequality \eqref{equa:errorest} is proved, while the inequality \eqref{equa:errorest2} can be easily derived from \eqref{equa:errorest} and the definition of $\eta$.
    \end{proof}

From the above analysis, a direct observation is that the testing space $\mathcal{V}$ should be large enough so that $u^\infty$ can be well approximated by $\mathcal{P}u^\infty$. To control the first error component in \eqref{equa:errorest}, $\Vert G_\Gamma(z,\cdot)\Vert_{L^2(\Gamma)}$ is desired to be small, which can be achieved by employing a regularization method for solving 
the system \eqref{equa:FEMf2}. On the other hand, an over-regularization may make the third error component in \eqref{equa:errorest} large. 
Thus, the role of the regularization is to balance the first and the third error components in \eqref{equa:errorest}. We now summarize the abstract algorithm for this framework in Algorithm \ref{Algor:FEMframe}. 

\textbf{Parallel implementation of the finite space framework.} Note that solving the liner system\eqref{equa:FEMf2} at different sampling points are independent of each other. Thus, the probing function can be computed in parallel, which is very important
for large-scale reconstructions. In addition, if all sampling points employ the same regularized inverse $\mathcal{R}_\sigma$ to compute the probing function via $\mathbf{F}(z)= \mathcal{R}_\sigma\mathbf{B}(z)$, then the regularized inverse $\mathcal{R}_\sigma$ 
needs to be computed only once.

\begin{algorithm}[h]
	\caption{Constructing the probing functions based on the finite space framework}
	\label{Algor:FEMframe}
	\KwIn{ \\
		\quad$\bullet$ Given data $u^\infty_{\Gamma,\delta}(\hat{x})$ measured over the partial set of curves 
		$\Gamma\subset \mathbb{S}^{1}$.\\
		\quad$\bullet$ A sampling domain $\Omega$.
	}
	Choose a trial space $\mathcal{U}$ and a testing space $\mathcal{V}$:
	\begin{equation}
		\mathcal{U}=\text{Span}\{\psi_1,\psi_2,\cdots,\psi_M\}\subset L^2(\Gamma),
	\end{equation} 
	\begin{equation}
		\mathcal{V}=\text{Span}\{\varphi_1,\varphi_2,\cdots,\varphi_N\}\subset C^\infty(\mathbb{S}^1).
	\end{equation}
	\\ Compute the matrix $\mathbb{A}\in \mathbb{C}^{N\times M }$ and the vector function $\mathbf{B}(z)\in  \mathbb{C}^N, z\in \Omega$ with entries
	\begin{equation}
		\mathbb{A}_{nm}=\langle \psi_m(\hat{x}),\varphi_n(\hat{x})\rangle_{L^2(\Gamma)}, \quad \mathbf{B}_n(z)=\langle G^\infty(z,\hat{x}), \varphi_n(\hat{x})\rangle_{L^2(\mathbb{S}^1)}
	\end{equation}
	for  $m=1,2,\cdots,M$ and $n=1,2,\cdots,N$.
	\\ Compute $\mathbf{F}(z)=(f_1(z),f_2(z),\cdots,f_M(z))^{T}$ by solving the equation $\mathbb{A}(\mathbf{F}(z))\approx \mathbf{B}(z)$ with a regularization method, e.g., the Tikhonov regularization: 
	\begin{equation}
		\mathbf{F}(z)=(\sigma I +\mathbb{A}^{*}\mathbb{A})^{-1}\mathbb{A}^{*}\mathbf{B}(z),\quad z\in \Omega.
  \label{equa:Tikhonov}
	\end{equation}
	\\ Form the probing function:
	\begin{equation}
	   G_\Gamma(z,\hat{x})=\sum_{m=1}^{M}f_m(z)\psi_m(\hat{x}), \quad z\in\Omega, \hat{x} \in\Gamma.
	\end{equation}
	\\Compute the index function: 
	\begin{equation}
		\mathcal{I}_\Gamma(z)= 	\big|\langle G_\Gamma(z,\hat{x}), u^\infty_{\Gamma,\delta}(\hat{x})\rangle_{L^2(\Gamma)}\big|,\quad z\in\Omega.
	\end{equation}
\end{algorithm}

In addition to the regularization, the trial space and testing space are two crucial components in the finite space framework.
In the next two subsections, we provide one choice of the trial space and two of the testing space.

\subsection{Finite Fourier space method}
\label{sec:FFSM}
It is known that each analytic function has exponentially decaying Fourier coefficients. 
Based on the approximation ability of the Fourier series, a natural choice of the trial and testing spaces 
is the finite Fourier space:
\begin{equation}
	\mathbf{F}_P:=\text{Span}\bigg\{\frac{1}{\sqrt{2\pi}}e^{-\text{i}P\theta},\frac{1}{\sqrt{2\pi}}e^{-\text{i}(P-1)\theta},\cdots,\frac{1}{\sqrt{2\pi}}e^{\text{i}(P-1)\theta},\frac{1}{\sqrt{2\pi}}e^{\text{i}P\theta}\bigg\}.
 \label{equa:FFS}
\end{equation} 
We shall call the resulting method as the finite Fourier space method (FFSM).
With this choice, the matrix $\mathbb{A}$ in \eqref{equa:FEMAB} can be computed exactly as 
\begin{equation}
	\mathbb{A}_{nm}=\frac{1}{2\pi}\langle e^{\text{i}m\theta},  e^{\text{i}n\theta}\rangle_{L^2(\Gamma)}=\begin{cases}
		\sum_{l=1}^{L}\frac{\alpha_l}{\pi}, \quad m=n,\\
		\sum_{l=1}^{L}\frac{\sin((m-n)\alpha_l)}{(m-n)\pi}e^{\text{i}(m-n)\beta_l}, \quad m\ne n.
	\end{cases}
\end{equation}
By the orthogonality of the Fourier basis and the Jacobi–Anger expansion\cite{colton2019inverse}:
\begin{equation}
	\begin{split}
		e^{\text{i}kz\cdot\hat{x}}&=J_0(k|z|)+2\sum_{n=1}^{\infty}\text{i}^n J_n(k|z|)\cos(n(\theta_z-\theta_x))\\
		&=\sum_{n=-\infty}^{\infty}\text{i}^n J_n(k|z|)e^{\text{i}n\theta_z}e^{-\text{i}n\theta_x},
	\end{split}	
 \label{equa:Jacobi}
\end{equation}
we obtain the vector $\textbf{B}_n(z)$ in \eqref{equa:FEMAB}:
\begin{equation}
	\textbf{B}_n(z)=\langle G^\infty(z,\hat{x}), \dfrac{1}{\sqrt{2\pi}}e^{\text{i}n\theta}\rangle_{L^2(\mathbb{S}^1)} =\frac{\text{i}^{-n}e^{\text{i}\pi/4}}{2\sqrt{k}} J_n(k|z|)e^{-\text{i}n\theta_z}, \quad n=-P,-(P-1),\cdots,P-1,P; z\in \Omega.
\end{equation}

\subsection{Finite source space method}
\label{sec:FSSM}
Recall that the far-field pattern has the following convolution representation
\begin{equation}
	u^\infty(\hat{x})=\int_\Omega G^\infty(y,\hat{x})I(y)dy\approx \sum_j w_j G^\infty(y_j,\hat{x}), \quad \hat{x} \in \mathbb{S}^1.
\end{equation}
Thus, given a set of finite points $\{y_j\}_{j=1}^{N}\subset \Omega$, which are dense enough in $\Omega$,
the far-field pattern can be well approximately by a function in the finite dimensional space:
\begin{equation}
	\text{Span}\{G^\infty(y_1,\hat{x}),G^\infty(y_2,\hat{x}),\cdots,G^\infty(y_N,\hat{x})\}. 
\end{equation} 
This motivates the choice of the above space as a testing space.  
For the trial space, we propose to still employ the finite Fourier space\eqref{equa:FFS}. In conclusion, we have
\begin{equation}
	\mathcal{U}=\text{Span}\bigg\{\frac{1}{\sqrt{2\pi}}e^{-\text{i}P\theta},\frac{1}{\sqrt{2\pi}}e^{-\text{i}(P-1)\theta},\cdots,\frac{1}{\sqrt{2\pi}}e^{\text{i}(P-1)\theta},\frac{1}{\sqrt{2\pi}}e^{\text{i}P\theta}\bigg\}, 
 \label{equa:fssmTest}
\end{equation}
\begin{equation}
		\mathcal{V}=\text{Span}\{G^\infty(y_1,\hat{x}),G^\infty(y_2,\hat{x}),\cdots,G^\infty(y_M,\hat{x})\}.
  \label{equa:fssmTria}
\end{equation}
Since each $G^\infty(y_j,\hat{x})$ can be considered as the far-field pattern generated by a source located at the point 
$y_j$, we call this method the finite source space method (FSSM). We see that the difference between the finite Fourier space method and the finite source space method is the choice of the testing spaces. 
By using the Jacobi-Anger expansion \eqref{equa:Jacobi}, we have the following expression for the matrix $\mathbb{A}$ in \eqref{equa:FEMAB}:

\begin{equation}
	\mathbb{A}_{nm}=\frac{1}{\sqrt{2\pi}}\langle e^{\text{i}m\theta}, G^\infty(y_n,\hat{x})\rangle_{L2(\Gamma)}=	\frac{e^{-\text{i}\pi/4}}{2\pi\sqrt{k}}\sum_{p=-\infty}^{\infty}\sum_{l=1}^{L}\text{i}^{p}e^{\text{i}(m-p)\beta_l}e^{\text{i}p\theta_{y_n}}J_p(k|y_n|)C_{l,m,p},
\end{equation}
where 
\begin{equation}
	C_{l,m,p}=\begin{cases}
		\alpha_l \quad m=p,\\
		\frac{\sin(\alpha_l(m-p))}{m-p} \quad m\ne p,
	\end{cases}
\end{equation}
and the following expression for $\textbf{B}_n(z)$\cite{li2013direct}:
\begin{equation}
	\textbf{B}_n(z) = \langle G^\infty(z,\hat{x}), G^\infty(y_n,\hat{x})\rangle_{L^2(\mathbb{S}^1)}=\frac{1}{4k} J_0(k|z-y_n|).
\end{equation}
A natural choice of the quadrature points $\{y_j\}_{j=1}^N$ is a set of uniformly distributed points in $\Omega$.

\section{A deep learning approach for constructing probing functions: deep probing network}\label{sec:deeplearning}
In this section, we exploit the deep learning technique to construct some effective probing functions to be used 
for the DSM when only limited-aperture data is available. Thanks to the universal approximation ability of the neural network, we can construct the probing function $G_\Gamma(z,\hat{x})$ by introducing a neural network $\mathcal{NN}_\vartheta(z,\hat{x})$, where $(z,\hat{x})$ is the input of the neural network and $\vartheta$ denotes the parameters in the neural network. Finding a reasonable network $\mathcal{NN}_\vartheta(z,\hat{x})$ as an effective probing function is 
then equivalent to finding a suitable $\vartheta$ from a parameter space $\Theta$. Thus, an important part is to design an appropriate network architecture; another crucial part is to introduce a proper loss function so that we can iteratively update $\vartheta$ with finite steps to get a satisfactory network  $\mathcal{NN}_\vartheta(z,\hat{x})$ as the probing function. Recall that the role of the loss function is to enforce the neural network to satisfy some desired properties. Similar to the methods based on the finite space framework, a reasonable desired property for the network is that for any far-field pattern $u^\infty(\hat{x})$,  it holds 
\begin{equation}
		\langle \mathcal{NN}_{\vartheta}(z,\hat{x}),u^\infty(\hat{x})\rangle_{L^2(\Gamma)}\approx \langle G^\infty(z,\hat{x}), u^\infty(\hat{x})\rangle_{L^2(\mathbb{S}^1)}, \quad \forall z\in \Omega.
  \label{equa:property}
\end{equation}
By using the linearity of the inner product and the convolution representation of the far-field pattern \eqref{equa:disIntegtal}, it would be reasonable to require the following in order to meet the approximation property \eqref{equa:property}:
\begin{equation}
		\langle \mathcal{NN}_{\vartheta}(z,\hat{x}),G^\infty(y,\hat{x})\rangle_{L^2(\Gamma)}\approx \langle G^\infty(z,\hat{x}), G^\infty(y,\hat{x})\rangle_{L^2(\mathbb{S}^1)}, \quad  \forall y, z\in \Omega,
\end{equation} 
or finding a network $\mathcal{NN}_{\vartheta}$ to meet the following approximation
for an appropriately selected integer $N$:
\begin{equation}
		\langle \mathcal{NN}_{\vartheta}(z,\hat{x}),\sum_{n=1}^{N}c_nG^\infty(y_n,\hat{x})\rangle_{L^2(\Gamma)}\approx \langle G^\infty(z,\hat{x}), \sum_{n=1}^{N}c_nG^\infty(y_n,\hat{x})\rangle_{L^2(\mathbb{S}^1)}. \quad  \forall y_n, z\in \Omega, \,c_n\in \mathbb{C}.
\end{equation}
But noting that $G^\infty(z,\hat{x})$ may be of small values, especially for high wavenumbers, 
we may encounter some numerical difficulties during the training process. By the linearity of the inner product and the expression \eqref{equa:farfieldG} 
of $G^\infty(z,\hat{x})$, 
it is reasonable for us to construct a neural network $\mathcal{NN}_{\vartheta}$ 
to meet the following approximation instead:
\begin{equation}
    \langle \mathcal{NN}_{\vartheta}(z,\hat{x}),\sum_{n=1}^{N}c_n\exp(-\text{i}k\hat{x}\cdot y_n)\rangle_{L^2(\Gamma)}\approx \langle \exp(-\text{i}k\hat{x}\cdot z), \sum_{n=1}^{N}c_n\exp(-\text{i}k\hat{x}\cdot y_n)\rangle_{L^2(\mathbb{S}^1)}, \quad  \forall y_n, z\in \Omega, \forall c_n\in \mathbb{C}.
    \label{equa:deepSystem}
\end{equation}
Thus, we can introduce the following loss function with continuous form:
\begin{equation}
	\frac{1}{M}\sum_{m=1}^{M}\Vert 	\langle \mathcal{NN}_{\vartheta}(z,\hat{x}),v^\delta_m(\hat{x})\rangle_{L^2(\Gamma)}- \langle \exp(-\text{i}k\hat{x}\cdot z), v_m(\hat{x})\rangle_{L^2(\mathbb{S}^1)} \Vert^2_{L^2(\Omega)}
\end{equation}
where $v_m(\hat{x})=\sum_{n=1}^{N}c_{nm}\exp(-\text{i}k\hat{x}\cdot y_{nm})$ and $v^\delta_m(\hat{x})$ is obtained by polluting $v_m(\hat{x})$ with noise to make the learned probing function robust to noise. In each iteration $\Re (c_{nm})$ and  $\Im (c_{nm})$ are randomly chosen from the normal distribution, $\{y_{nm}\}$ are randomly chosen from the uniform distribution of $\Omega$.  To compute the loss function numerically, we introduce the following discrete loss function:
\begin{equation}
	Loss(\vartheta) = \frac{1}{ML}\sum_{m=1}^{M}\sum_{l=1}^{L}\bigg|\frac{|\Gamma|}{Q}\sum_{q=1}^{Q}\mathcal{NN}_{\vartheta}(z_l,\hat{x}_q)\overline{v_m^\delta(\hat{x}_q)}-2\pi\sum_{n=1}^{N}\overline{c_{nm}}J_0(k|z_l-y_{nm}|)\bigg|^2,
 \label{equa:lossfunc}
\end{equation}
where $L$ is a fixed number and points $\{z_l\}_{l=1}^{L}$ are randomly chosen from the uniform distribution of $\Omega$ in each iteration.  $\{\hat{x}_q\}_{q=1}^{Q}$ are some uniformly distributed points on $\Gamma$. We name the proposed deep learning scheme ``deep probing network" and summarize it in Algorithm \ref{Algor:learning}.  
\begin{algorithm}[h]
	\caption{Constructing the probing function based on the deep probing network}
	\label{Algor:learning}
	\KwIn{ \\
		\quad$\bullet$ Given the measurement curves $\Gamma$ and the sampling domain $\Omega$.
	}
	\SetKwInput{kwInit}{Preparations}
	\SetKwInput{sss}{In each iteration, we apply the following procedures}
	\kwInit{\\
	\quad$\bullet$ Choose positive integers $M,N,Q,L,N_{\text{ite}}$, and a positive number $\lambda$.\\
	\quad$\bullet$ Choose learning rates $\{\tau_j\}_{j=1}^{N_{\text{ite}}}$ and uniformly distributed points $\{\hat{x}_q\}_{q=1}^Q$ on $\Gamma$.\\
	\quad$\bullet$ Introduce a neural network $\mathcal{NN}_{\vartheta}$ with parameter $\vartheta$.}
	\sss{}
	For $n=1,2,\cdots,N;m=1,2,\cdots,M$, randomly sample the points $\{y_{nm}\}$ over 
	$\Omega$ in uniform distribution, 
	and the numbers $\{c_{nm}\}$ with $\Re (c_{nm})$ and  $\Im (c_{nm})$ following the normal distribution.
	\\ Define $M$ functions $\{v_m\}_{m=1}^M$ as
	\begin{equation}
		v_m(\hat{x})=\sum_{n=1}^{N}c_{nm}\exp(-\text{i}k\hat{x}\cdot y_{nm}), \quad m=1,2,\cdots,M;\hat{x}\in\mathbb{S}^1.
	\end{equation}
	\\ Randomly sample the noise level $\delta$ from the uniform distribution $\mathcal{U}(0,\lambda)$, and pollute $\{v_m\}_{m=1}^M$ via
	\begin{equation}
		v_m^\delta(\hat{x})=v_m(x)+\delta(\eta_r(\hat{x})+\text{i}\eta_i(\hat{x}))\frac{\Vert v_m\Vert_{L^2(\Gamma)}}{|\Gamma|^{1/2}},
	\end{equation}
	where $\eta_r(\hat{x})$ and $\eta_i(\hat{x})$ follow the normal distribution.
	\\ Randomly sample $L$ points $\{z_l\}_{l=1}^{L}$ from the uniform distribution of $\Omega$.
	\\Compute the loss function as 
	\begin{equation}
		Loss(\vartheta) = \frac{1}{ML}\sum_{m=1}^{M}\sum_{l=1}^{L}\bigg|\frac{|\Gamma|}{Q}\sum_{q=1}^{Q}\mathcal{NN}_{\vartheta}(z_l,\hat{x}_q)\overline{v_m^\delta(\hat{x}_q)}-2\pi\sum_{n=1}^{N}\overline{c_{nm}}J_0(k|z_l-y_{nm}|)\bigg|^2,
	\end{equation}
	\\ Update the parameters of the network at the $j$th iteration:
	\begin{equation}
		\vartheta\leftarrow \vartheta-\tau_j\nabla_\vartheta Loss(\vartheta).
	\end{equation}
\end{algorithm}

We now propose a special architecture for the network $\mathcal{NN}_{\vartheta}(z,\hat{x})$, which will be used in the numerical experiments. In this architecture, we construct the probing function $G_\Gamma(z,\hat{x})$ as
\begin{equation}
	G_\Gamma(z,\hat{x})=\sum_{n=-P}^{P}f_{\vartheta,n}(z)e^{\text{i}n\theta_x}+\exp(-\text{i}k\hat{x}\cdot z),
 \label{equa:network}
\end{equation}
where $\{f_{\vartheta,n}(z)\}_{n=-P}^{P}$ represent the output of a neural network with $z$ as the input. The architecture is visually shown in Fig.\ref{fig: network}. The output dimension of the neural network is equal to $4P+2$, where $2P+1$ elements represent real part of $\{f_{\vartheta,n}\}_{n=-P}^{P}$, and the others represent  imaginary part of $\{f_{\vartheta,n}\}_{n=-P}^{P}$. The term $\exp(-\text{i}k\hat{x}\cdot z)$ can be considered an initial guess of the desired probing function.

\begin{remark}
	The proposed network architecture \eqref{equa:network} shown in Fig.\ref{fig: network}, which only takes $z$ as the input of the neural network, is a special architecture, similar to the POD-DeepONet proposed in \cite{lu2022comprehensive}. Due to the severe ill-posedness of the problem with limited-aperture data, the high-frequency modes of the probing function become important to capture the high-frequency information of the limited-aperture data.  Our numerical experiences show that this special architecture can achieve better results than a general fully connected neural network that takes $(z,\hat{x})$ as the input and $G_\Gamma(z,\hat{x})$ as the output. This is because general neural networks have limitations on learning high-frequency modes of a function\cite{CiCP-28-1746}, while the proposed network architecture \eqref{equa:network} employ the explict Fourier basis $\{e^{\text{i}n\theta_x}\}_{n=-P}^P$ to capture the high-frequency modes so that only the coefficient functions $\{f_{\theta,n}\}_{n=-P}^P$ are required to be learned. However, we still keep the general notation $\mathcal{NN}_\vartheta(z,\hat{x})$ in this method to allow more flexibility for the network architecture. For example, a more general framework is to introduce an architecture with the form $G_\Gamma(z,\hat{x})=\sum_{n=1}^{N}f_{\vartheta_1,n}(z)\psi_{\vartheta_2,n}(\hat{x})+\exp(-\text{i}k\hat{x}\cdot z)$, where $F_\vartheta(z)=\{f_{\vartheta_1,1}(z),f_{\vartheta_1,2}(z),\cdots,f_{\vartheta_1,N}(z) \}$ and $\Psi_\vartheta(x)=\{\psi_{\vartheta_2,1}(x),\psi_{\vartheta_2,2}(x),\cdots,\psi_{\vartheta_2,N}(x) \}$ are two independent networks. This architecture shares a similar idea to the DeepONet\cite{lu2021learning}, where $F_{\vartheta_1}$ is called the branch net while $\Psi_{\vartheta_2}$ is called the trunk net.
\end{remark}

It is important to note that no data is required to learn the probing function during the training process, i.e., it is an unsupervised learning method. Moreover, the method does not pose any conditions on the incident fields except on the wavenumber, 
which can also be relaxed by observing the fact that a neural network satisfying equation \eqref{equa:deepSystem} 
can also meet the following properties for a new wavenumber $\tilde{k}$ and $\forall \frac{\tilde{k}}{k}y_n, \frac{\tilde{k}}{k}z\in \Omega$: 
\begin{equation}
\begin{split}
     \langle \mathcal{NN}_{\vartheta}(\frac{\tilde{k}}{k}z,\hat{x}),\sum_{n=1}^{N}c_n\exp(-\text{i}\tilde{k}\hat{x}\cdot y_n)\rangle_{L^2(\Gamma)}
     =&\langle \mathcal{NN}_{\vartheta}(\frac{\tilde{k}}{k}z,\hat{x}),\sum_{n=1}^{N}c_n\exp(-\text{i}k\hat{x}\cdot \frac{\tilde{k}}{k}y_n)\rangle_{L^2(\Gamma)}\\
     \approx&\langle \exp(-\text{i}k\hat{x}\cdot\frac{\tilde{k}}{k}z),\sum_{n=1}^{N}c_n\exp(-\text{i}k\hat{x}\cdot \frac{\tilde{k}}{k}y_n)\rangle_{L^2(\mathbb{S}^1)}\\
     =& \langle \exp(-\text{i}\tilde{k}\hat{x}\cdot z), \sum_{n=1}^{N}c_n\exp(-\text{i}\tilde{k}\hat{x}\cdot y_n)\rangle_{L^2(\mathbb{S}^1)}.
\end{split}
\end{equation}
Thus, by denoting $\widetilde{\mathcal{NN}}_{\vartheta}(z,\hat{x}):=\mathcal{NN}_{\vartheta}(\frac{\tilde{k}}{k}z,\hat{x})$, $\widetilde{\mathcal{NN}}_{\vartheta}(z,\hat{x})$ can be a probing function for wavenumber $\tilde{k}$ with the sampling domain $\widetilde{\Omega}=:\{z|\frac{\tilde{k}}{k}z\in\Omega\}$. These properties also hold for the methods based on the finite space framework.
\begin{figure}[htp]
		\centering
	\includegraphics[width=0.7\linewidth]{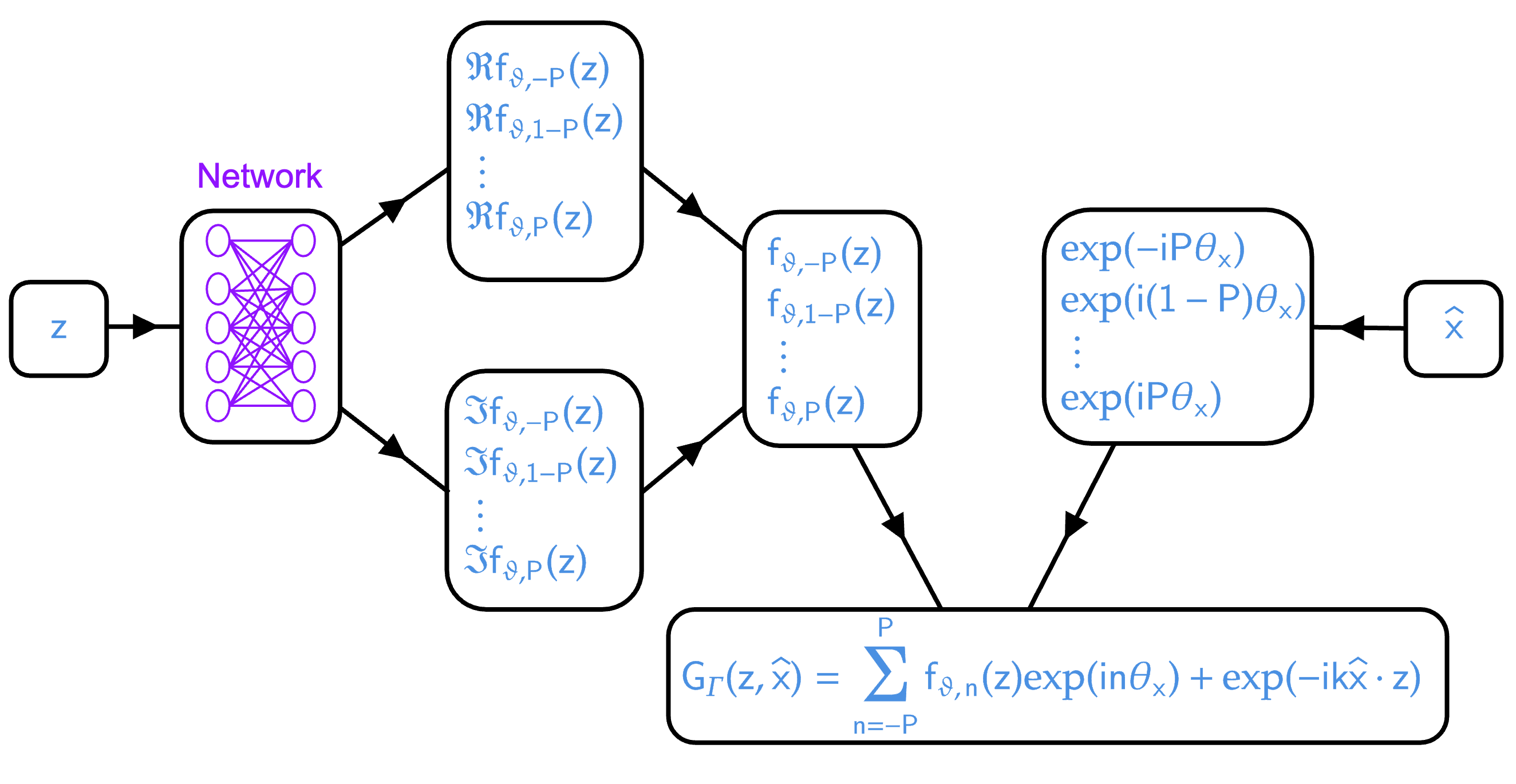}
	\caption{The architecture of the neural network used in the numerical experiments.}
	\label{fig: network}
\end{figure}

\textbf{Parallel implementation with multiple deep probing networks.} For the case where the sampling domain $\Omega$ is relatively large, employing a single neural network to learn the probing function for the whole sampling domain $\Omega$ may face some optimization challenges despite the universal approximation ability of the neural networks. To overcome this limitation, we propose to divide the whole sampling domain $\Omega$ into several smaller subdomains $\{\Omega_i\}_{i=1}^K$ with $\Omega = \cup_{i=1}^K \Omega_i$, and introduce $K$ deep probing networks $\{\mathcal{NN}_{\vartheta_i}(z,\hat{x}), z\in \Omega_i\}_{i=1}^K$. Then we suggest the following loss function corresponding to \eqref{equa:lossfunc} for the multiple deep probing networks:
\begin{equation}
    Loss(\vartheta_i) = \frac{1}{ML}\sum_{m=1}^{M}\sum_{l=1}^{L}\bigg|\frac{|\Gamma|}{Q}\sum_{q=1}^{Q}\mathcal{NN}_{\vartheta_i}(z_l,\hat{x}_q)\overline{v_m^\delta(\hat{x}_q)}-2\pi\sum_{n=1}^{N}\overline{c_{nm}}J_0(k|z_l-y_{nm}|)\bigg|^2, \quad i=1,2,\cdots,K,
\end{equation}
where points $\{z_l\}_{l=1}^{L}$ are randomly chosen from the uniform distribution of $\Omega_i$ in each iteration, whereas $\{y_{nm}\}$ are still randomly chosen from the uniform distribution of $\Omega$. We notice that each network is trained independently of the other, thus the multiple networks can be trained in parallel.

\section{Numerical Experiments}
\label{sec:numeri}
In this section, we present several representative numerical experiments to illustrate and compare the performance of different methods proposed in this paper. We consider two different configurations for receivers to measure the far-field data. In the first configuration, 100 receivers are uniformly located between the angle interval $[-\frac{2}{5}\pi,\frac{2}{5}\pi]$ as shown in Fig.\ref{fig:Config1}.  In the second configuration, 90 receivers are uniformly located in the angle intervals $[-\frac{1}{8}\pi,\frac{1}{8}\pi]\cup [(\frac{2}{3}-\frac{1}{8})\pi,(\frac{2}{3}+\frac{1}{8})\pi]\cup[(-\frac{2}{3}-\frac{1}{8})\pi,(-\frac{2}{3}+\frac{1}{8})\pi]$ as shown in Fig.\ref{fig:Config2}.  
\begin{figure}[htp]
	\centering
	\begin{subfigure}{0.4\textwidth}
		\centering
		\includegraphics[width=0.9\linewidth]{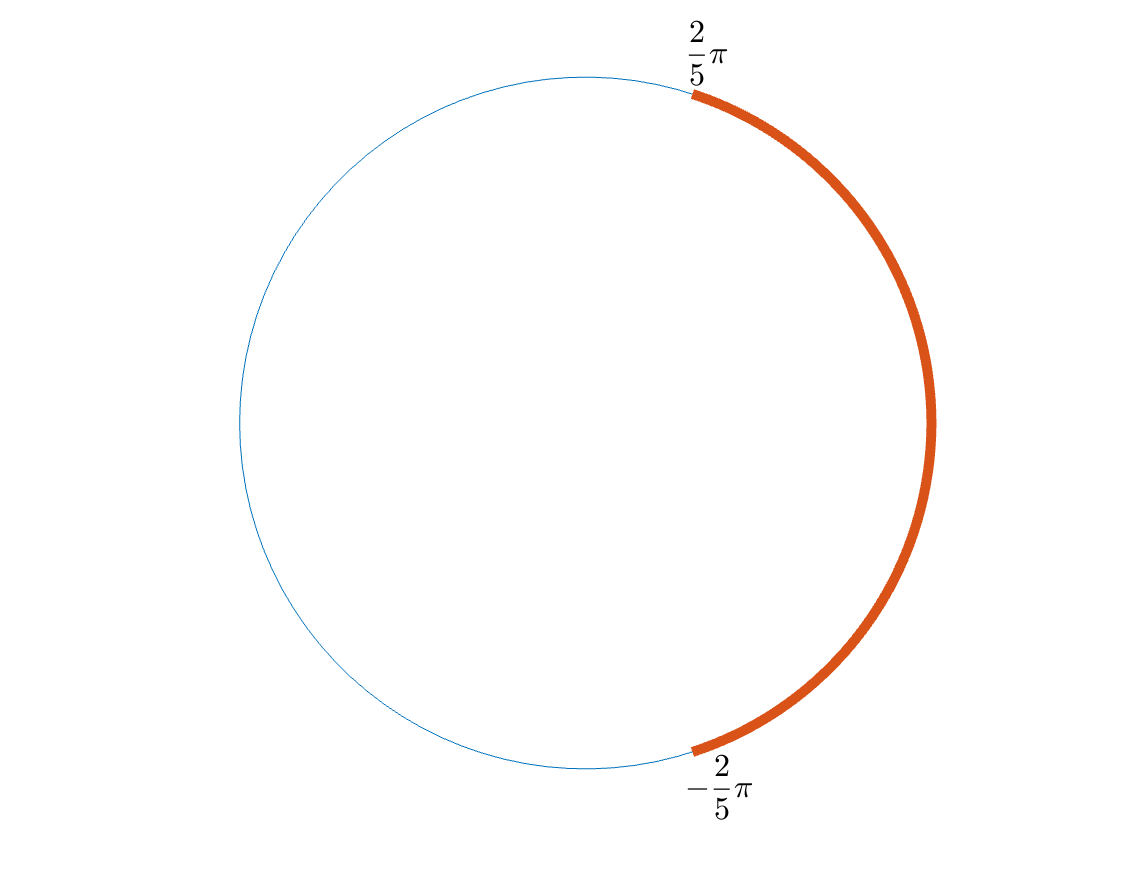}
		\caption{Configuration \uppercase\expandafter{\romannumeral1}}
		\label{fig:Config1}
	\end{subfigure}
		\begin{subfigure}{0.4\textwidth}
		\centering
		\includegraphics[width=0.95\linewidth]{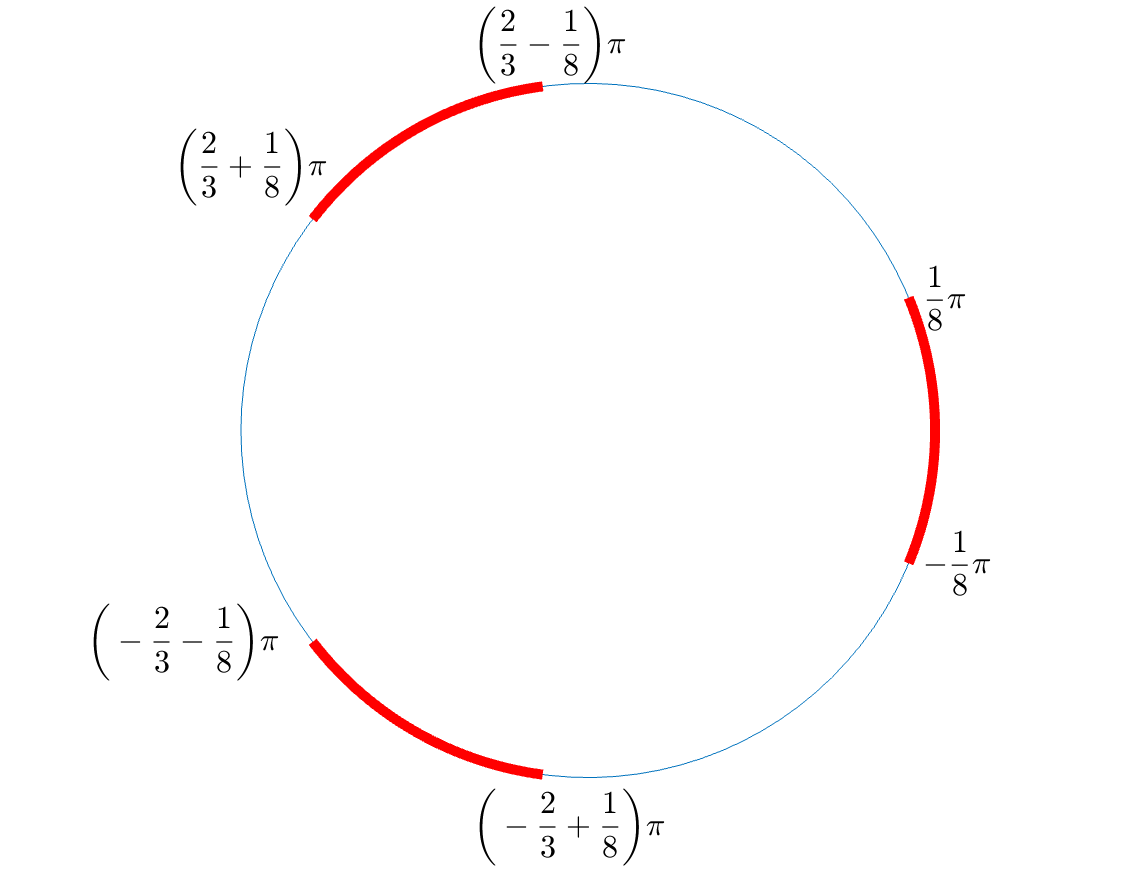}
		\caption{Configuration \uppercase\expandafter{\romannumeral2}}
		\label{fig:Config2}
	\end{subfigure}
	\caption{Two different configurations considered in the numerical experiments, where the highlighted red parts denote the measurement area $\Gamma$.}
	\label{fig:Config}
\end{figure}

The wavenumber is chosen as $k=8$ and the sampling domain is $\Omega=[-1,1]\times[-1,1]$. The noisy data $u^\infty_\delta$ are generated point-wisely by 
\begin{equation}
	u^\infty_\delta(\hat{x})=u^\infty(\hat{x})+\delta(\eta_r(\hat{x})+\text{i}\eta_i(\hat{x}))\frac{\Vert u^\infty\Vert_{L^2(\Gamma)}}{|\Gamma|^{1/2}}  \quad \hat{x} \in \Gamma,
\end{equation}
where both $\eta_r(\hat{x})$ and $\eta_i(\hat{x})$  follow the standard normal distribution, and $\delta$ is the noise level.

We will employ the following methods with specified details to compute the index functions:
\begin{itemize}
	\item \textbf{$G^\infty$+Full Data}: As a comparison, we compute the index function with full-aperture data by using the classical DSM, i.e., employing $G^\infty(z,\hat{x})$ as the probing function. 
	
	\item \textbf{$G^\infty$+Partial Data}: In this method, we directly employ $G^\infty(z,\hat{x})$ as the probing function to compute the index function with limited-aperture data measured from $\Gamma$.
	
	\item \textbf{Deep Probing Network}: This method employs the deep probing network proposed in Section \ref{sec:deeplearning} to construct the probing function over the measurement curves $\Gamma$. We employ the network architecture shown in Fig.\ref{fig: network} with $P=20$ to approximate the probing function. A fully connected network with 6 layers $[2,200,200,200,200,4P+2]$ and $ReLU$ activation function is used. In Algorithm \ref{Algor:learning}, we set $M=400,N=3,L=400,N_{\text{ite}}=5000, \lambda=0.05$, as well as $Q=100$ for configuration \uppercase\expandafter{\romannumeral1} and $Q=90$ for configuration \uppercase\expandafter{\romannumeral2}. The Adam optimizer is used to update the parameters of the neural network. The learning rate starts at 0.005 and decreases by a factor of 0.9 every 100 iterations. 

	\item \textbf{FFSM$\mathtt{m}$}: This method employs the finite Fourier space method proposed in Section \ref{sec:FFSM} to construct the probing function over the measurement curves $\Gamma$. We set $P=20$ in equation \eqref{equa:FFS} for both trial and testing spaces. We use the Tikhonov regularization $(\sigma I +\mathbb{A}^{*}\mathbb{A})^{-1}\mathbb{A}^{*}$ with $\sigma=0.1^{\mathtt{m}}$ and we will present the results with different choices of $\mathtt{m}$.
	
	\item \textbf{FSSM$\mathtt{m}$}: This method employs the finite source space method proposed in section \ref{sec:FSSM} to construct the probing function over the measurement surves $\Gamma$. We set $P=20$ for the trial space \eqref{equa:fssmTria}. The source points $\{y_j\}_{j=1}^{M}$ are equispaced grid points in $\Omega$ with $M=20^2$ for the testing space \eqref{equa:fssmTest}. 
We use the Tikhonov regularization $(\sigma I +\mathbb{A}^{*}\mathbb{A})^{-1}\mathbb{A}^{*}$ with $\sigma=0.1^{\mathtt{m}}$ and 
present the results with different choices of $\mathtt{m}$.
\end{itemize}
When ${N_{\text{inc}}}>1$ incidences are employed, we compute the index function averagely:
\begin{equation}
    \mathcal{I}_{\text{ave}}(z)=\frac{1}{N_{\text{inc}}}\sum_{j=1}^{N_{\text{inc}}}\mathcal{I}_j(z),
\end{equation}
where $\mathcal{I}_j(z)$ denotes the index function corresponding to the $j_{\text{th}}$ incidence.
For all the methods, the presented index functions are further normalized as
\begin{equation}
    \hat{\mathcal{I}}(z)=\frac{\mathcal{I}(z)}{\max_{z\in\Omega}\mathcal{I}(z)},\quad z\in \Omega.
\end{equation}

\begin{figure}[htp]
      \centering
	 \begin{subfigure}{0.85\textwidth}
		\centering
		 \includegraphics[width=\linewidth]{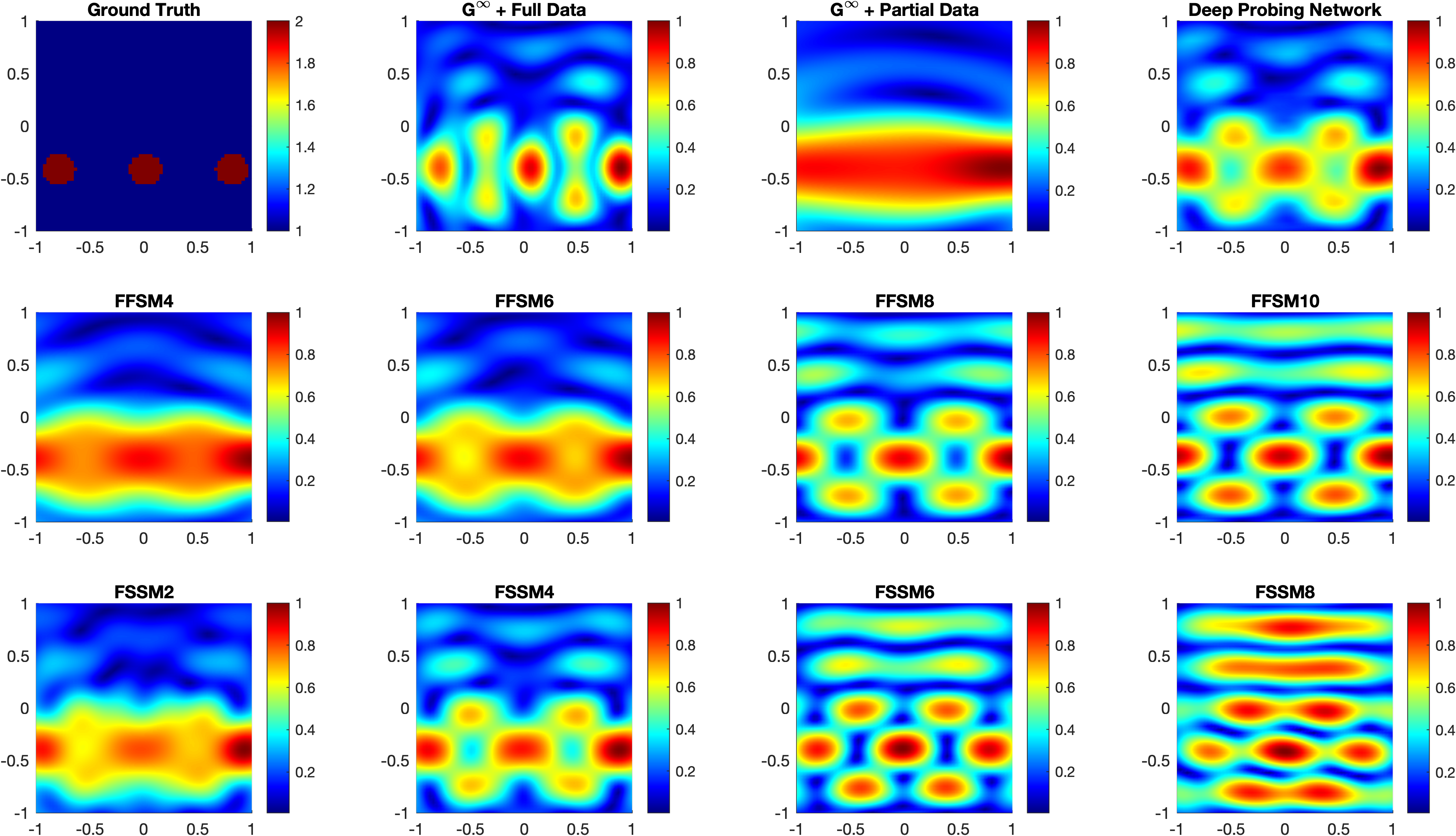}
		\caption{Reconstructions for \textbf{Example 1.1,} with noise level $\delta=1\%$}
  \vspace{0.2cm}
		\label{fig:exam6111}
	\end{subfigure}
	\begin{subfigure}{0.85\textwidth}
		\centering
		\includegraphics[width=\linewidth]{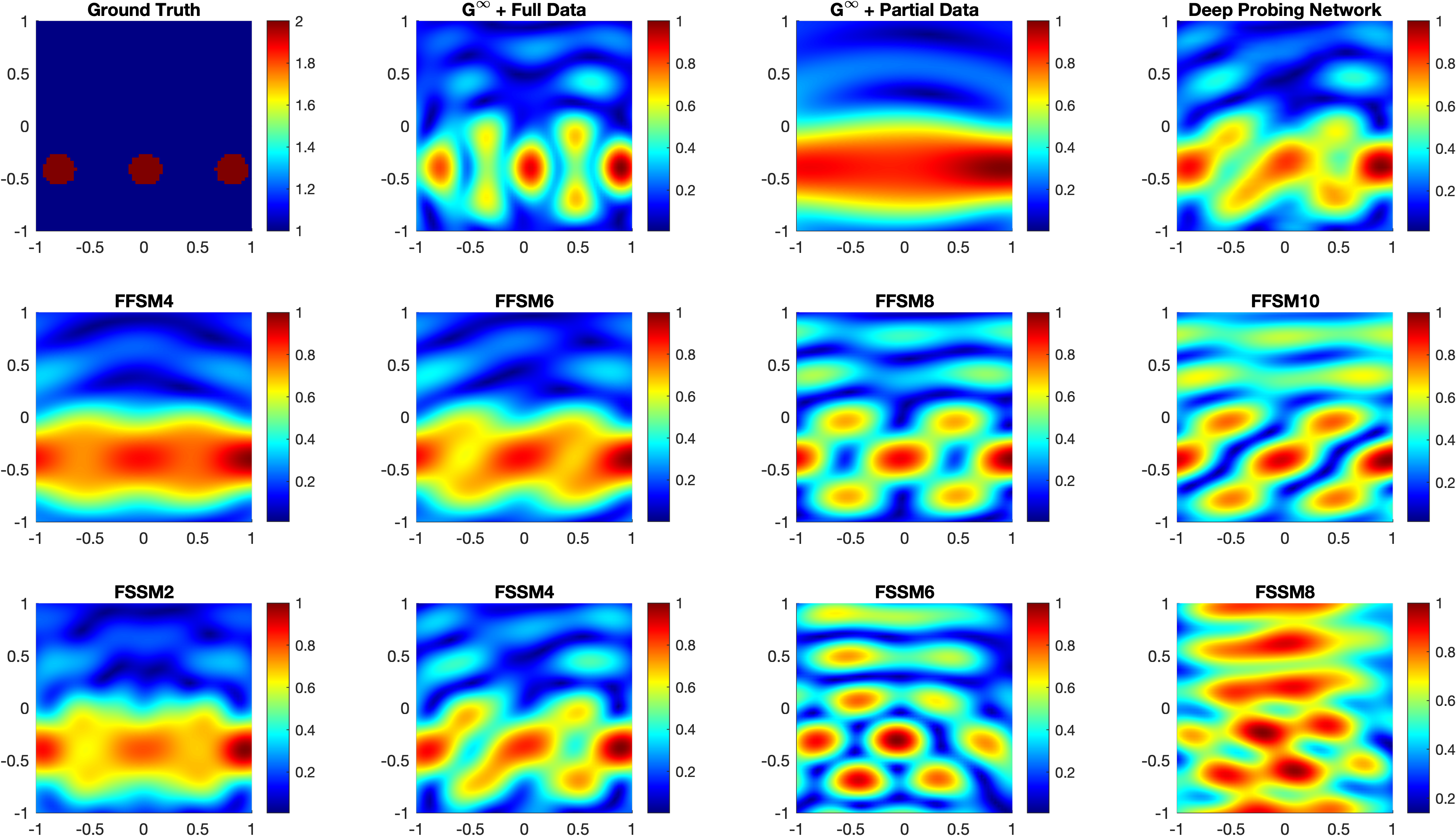}
		\caption{Reconstructions for \textbf{Example 1.1.} with noise level $\delta=5\%$}
		\label{fig:exam6115}
	\end{subfigure}
	\caption{Reconstructions for \textbf{Example 1.1,} with different noise levels.}
 \vspace{-0.5cm}
\end{figure}
\subsection{Numerical results for Configuration \uppercase\expandafter{\romannumeral1}}\label{Config1}

We first present two numerical examples for Configuration \uppercase\expandafter{\romannumeral1} shown in Fig.\ref{fig:Config1}.

\textbf{Example 1.1.}  In this example, three circles with radius $0.15$ located at $(-0.8,-0.4), (0,-0.4)$ and $(0.8,-0.4)$ are used to simulate inhomogeneous scatterers with refractive index 2.0. One incidence plane wave $u^i(x)=e^{\text{i}kx\cdot d}$ with $d=(1,0)$ is employed.

The reconstructions for this example with noise level $\delta=1\%$ are presented in Fig.\ref{fig:exam6111}, where the first image in the first row shows the exact image. It can be seen that the classical DSM, i.e., using $G^\infty(z,\hat{x})$ as the probing function, with full data, can provide a pretty satisfactory approximation for the locations of the scatterers and can distinguish the scatterers very well. However, due to the ``shadow region" phenomenon, although ``$G^\infty+$Partial Data" can provide a reasonable approximation for the vertical locations of the scatterers, it can only offer low-resolution reconstruction along the horizontal direction and can not distinguish and separate the three scatterers. Other figures show the reconstructions by different methods proposed in this work with limited-aperture data. From the last figure in the first row, we can see that the deep probing network can provide a much sharper reconstruction than ``$G^\infty+$Partial Data" and can distinguish the scatterers very well, which means that it can break the limitation of ``shadow region" to some extent. The other rows show the reconstructions by the finite Fourier space method (FFSM) and the finite source space method (FSSM) with different regularization parameters. It can be seen that with a proper choice of the regularization parameter, all these methods are also able to provide reasonable reconstructions. Specifically, the method FFSM achieves the best results with $\sigma=0.1^8$ among the parameter set $\{0.1^4,0.1^6,0.1^8,0.1^{10}\}$, and the method FSSM achieves the best results with $\sigma=0.1^4$ among the parameter set $\{0.1^2,0.1^4,0.1^6,0.1^8\}$. For large $\sigma$ (corresponding to images in the left), the resolution of the reconstructions by these methods is low and tends to behave similarly to ``$G^\infty+$Partial Data". For small $\sigma$,  these methods can not provide a more reliable approximation for the locations of the scatterers.

\begin{figure}[htp]
		\centering
	\includegraphics[width=0.7\linewidth]{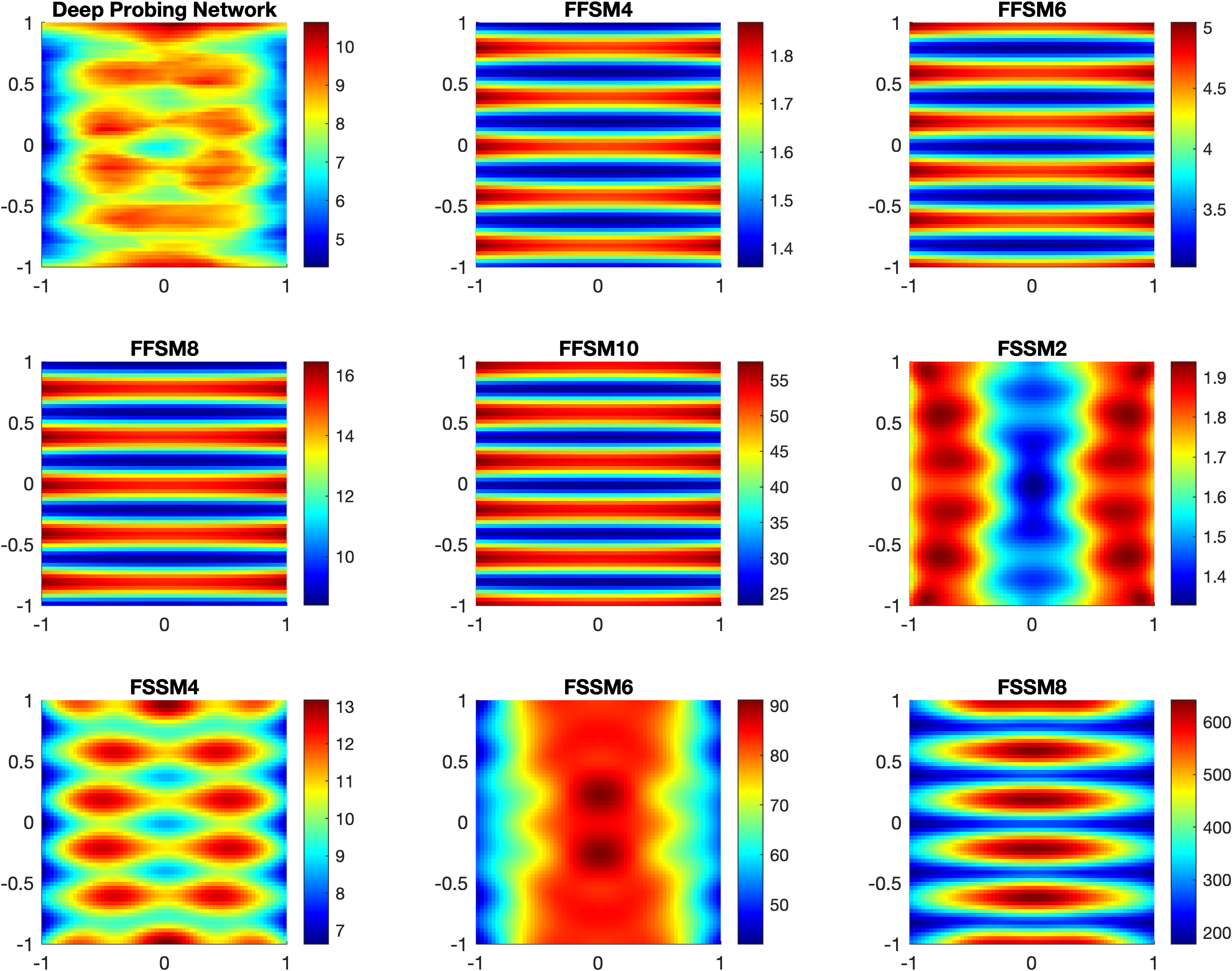}
	\caption{The relative norm function $RN(z)=\frac{\Vert G_\Gamma(z,\cdot)\Vert_{L^2(\Gamma)}}{\Vert G^\infty(z,\cdot)\Vert_{L^2(\Gamma)}}$ of the probing functions $G_\Gamma(z,\hat{x})$ constructed by different methods.}
	\label{fig: ProbingFunc}
\end{figure}

 To further test the robustness of different methods, we now consider the data with larger noise, $\delta=5\%$, and the recovered images are shown in Fig.\ref{fig:exam6115}. As the probing function $G^\infty(z,\hat{x})$ for classical DSM has relatively small norm and is mainly dominated by low-frequency modes if $k|z|$ is not too large, the reconstructions by ``$G^\infty+$Full Data" and ``$G^\infty+$Partial Data" with $5\%$ noise level are almost the same as those with $1\%$ noise level. However, as the lack of data makes the problem severely ill-posed,  in high noise level cases, we can see some distortions of the reconstructions by the deep probing network and other two methods with small regularization parameter $\sigma$. With large $\sigma$, we notice that the two regularization methods are stable to high-level noise but meet the limitation of low-resolution reconstruction. In Fig.\ref{fig: ProbingFunc}, we present the relative norm function $RN(z)=\frac{\Vert G_\Gamma(z,\cdot)\Vert_{L^2(\Gamma)}}{\Vert G^\infty(z,\cdot)\Vert_{L^2(\Gamma)}}$ of the probing functions $G_\Gamma(z,\hat{x})$ constructed by different methods. We can see that the norm of the probing function constructed by the deep probing network is around 5 to 10 times larger than the norm of $G^\infty(z,\hat{x})$. For the methods FFSM and FSSM, small $\sigma$ leads to the extremely large norm of the constructed probing function, which can then make the first error component in \eqref{equa:errorest} large as discussed in Lemma \ref{lemma:41}. Probing functions constructed by relatively large $\sigma$ have a relatively small norm, which can make the first error component in \eqref{equa:errorest} small but at the same time, the third error component in \eqref{equa:errorest} would be very large. Thus, to achieve high-resolution and reliable reconstructions, choosing a proper regularization parameter for the two methods is essential. An attractive advantage of the deep probing network is that it can escape the need to select the regularization parameter while the training process requires much more computational effort.

\begin{figure}[htp]
\centering
	\begin{subfigure}{0.9\textwidth}
		\centering
		\includegraphics[width=\linewidth]{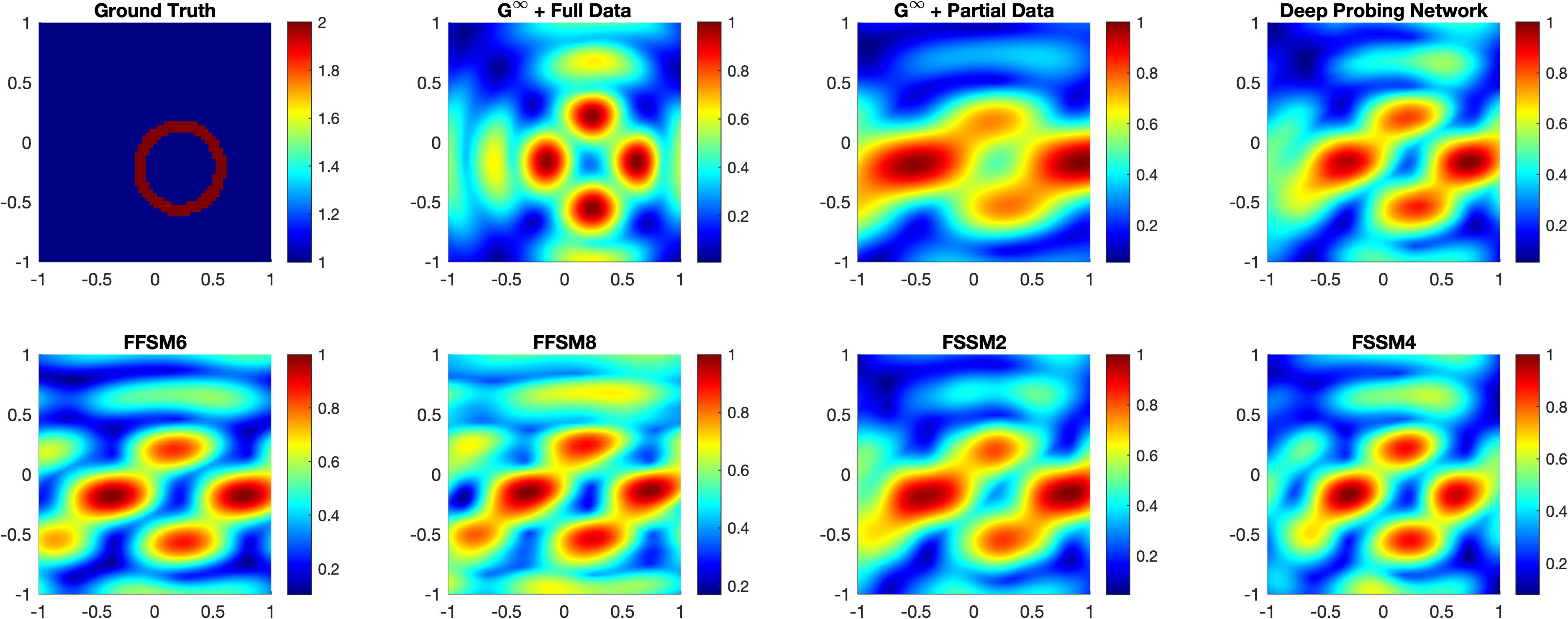}
		\caption{Reconstructions for \textbf{Example 1.2.} with noise level $\delta=1\%$}
   \vspace{0.5cm}
	\end{subfigure}
	\begin{subfigure}{0.9\textwidth}
		\centering
		\includegraphics[width=\linewidth]{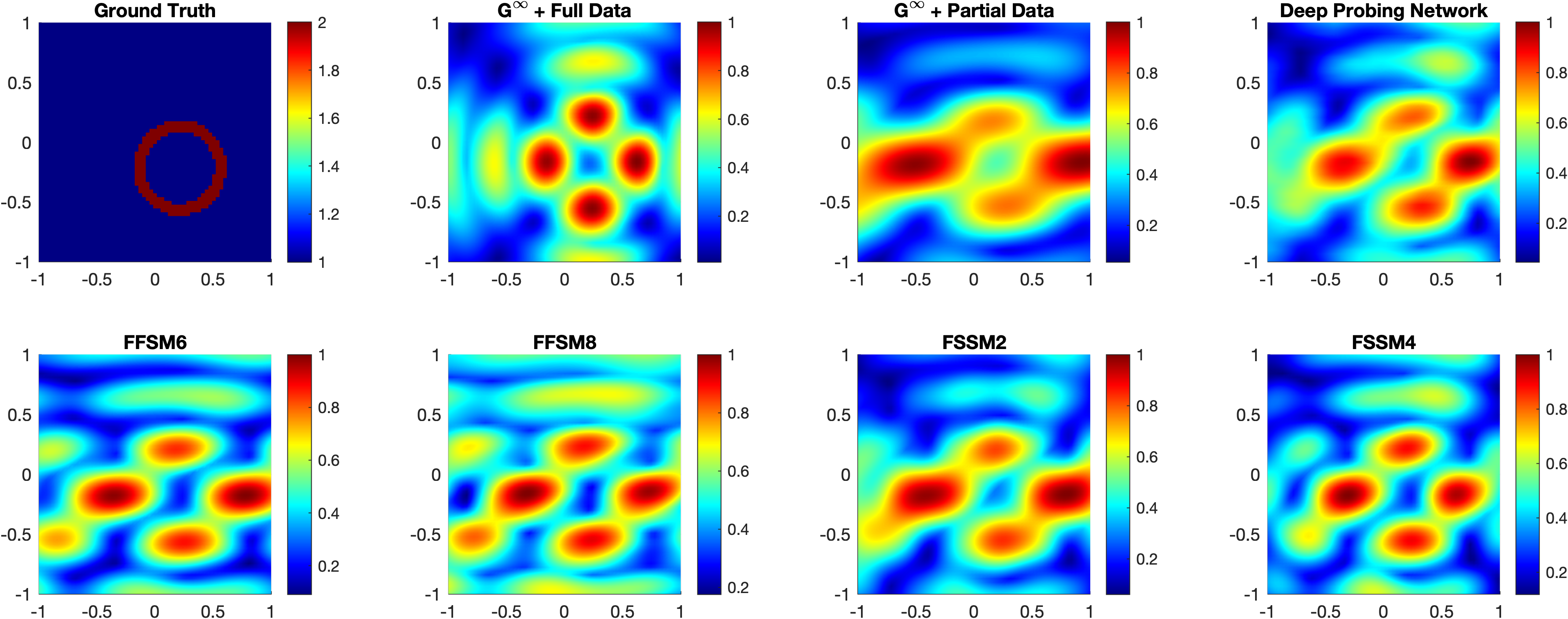}
		\caption{Reconstructions for \textbf{Example 1.2.} with noise level $\delta=5\%$}
	\end{subfigure}
	\caption{Reconstructions for \textbf{Example 1.2.} with different noise levels.}
	\label{fig:Example612}
\end{figure}

\textbf{Example 1.2.}  In this example, a scatterer with refractive index 2.0  is simulated by a ring located at the $(0.2,-0.2)$, with the inner and outer radius being 0.3 and 0.4, respectively. We employ two incidence plane waves with directions $d_1=(1,0)$ and $d_2=(0,1)$, respectively.

Similar to \textbf{Example 1.1.}, the recovered images presented in Fig.\,\ref{fig:Example612} show that the ``$G^\infty+$Partial Data" still faces the limitation of ``shadow region" with the resolution along the horizontal direction is low. The proposed deep probing network can alleviate the limitation stably. With the proper choice for the regularization parameter, the methods FFSM and FSSM can also provide better reconstructions than ``$G^\infty+$Partial Data", and the method FSSM with $\sigma=0.1^4$ seems to achieve the best results among the presented recovered images with limited-aperture data for this example.

\subsection{Numerical results for Configuration \uppercase\expandafter{\romannumeral2}}

We now present two numerical examples for Configuration \uppercase\expandafter{\romannumeral2} shown in Fig.\ref{fig:Config2}.

\textbf{Example 2.1.}  In this example, two circles with radius $0.15$ located at $(-0.6,-0.6)$ and $ (-0.2,-0.2)$ are used to simulate inhomogeneous scatterers with refractive index 2.0. One incidence plane wave $u^i(x)=e^{\text{i}kx\cdot d}$ with $d=(1,0)$ is employed.  The reconstructions for this example with different noise levels are presented in Fig.\ref{fig.exam621}.

\begin{figure}[htp]
\centering
	\begin{subfigure}{0.9\textwidth}
		\centering
		\includegraphics[width=\linewidth]{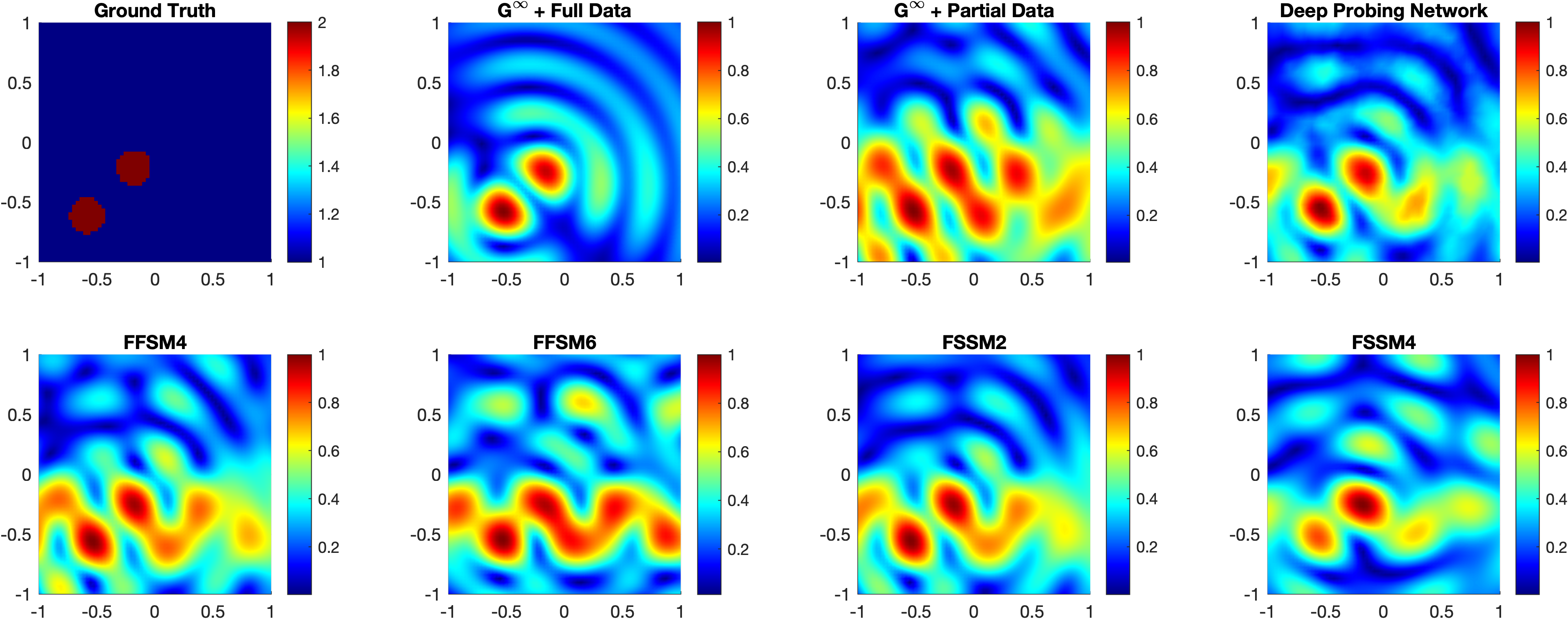}
		\caption{Reconstructions for \textbf{Example 2.1.} with noise level $\delta=1\%$}
		\label{fig: Logrestic Regression}
   \vspace{0.5cm}
	\end{subfigure}
	\begin{subfigure}{0.9\textwidth}
		\centering
		\includegraphics[width=\linewidth]{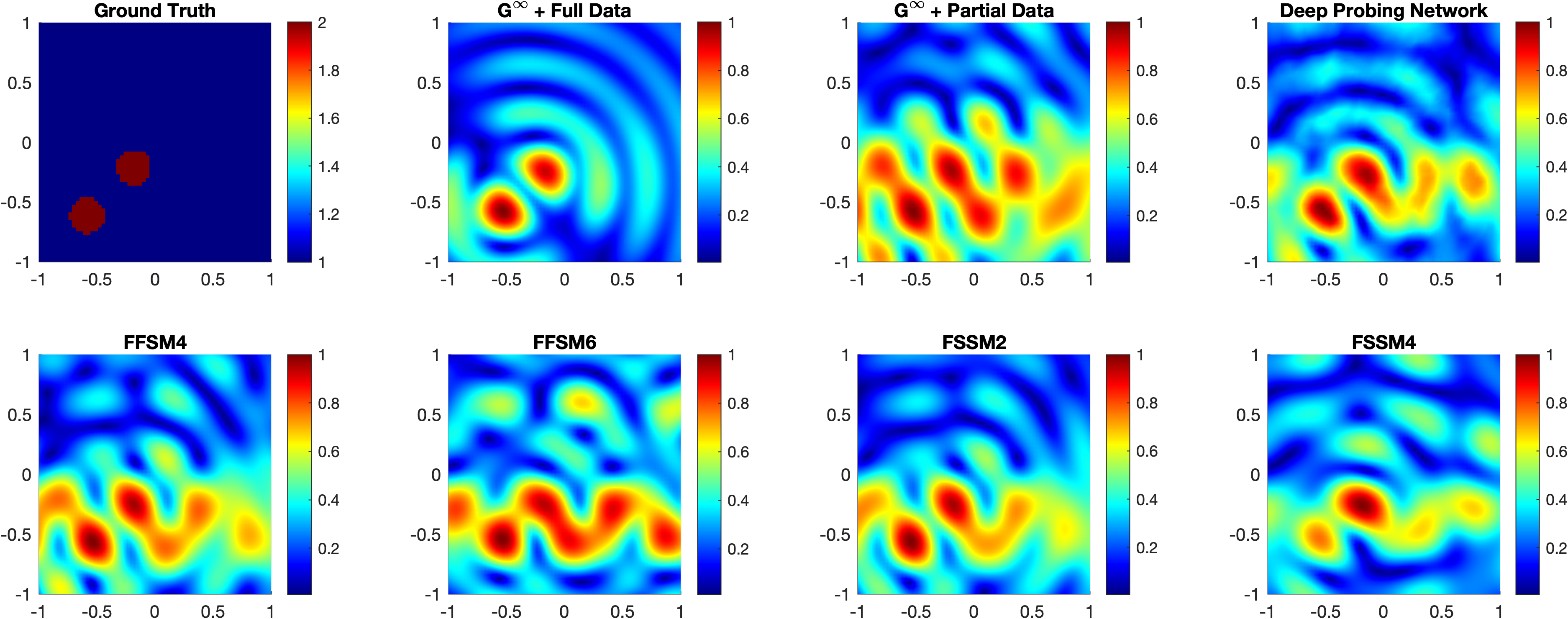}
		\caption{Reconstructions for \textbf{Example 2.1.} with noise level $\delta=5\%$}
		\label{fig: Logrestic Regression}
	\end{subfigure}
	\caption{Reconstructions for \textbf{Example 2.1.} with different noise levels.}
	\label{fig.exam621}
\end{figure}

With only one incidence, the DSM with full data can distinguish the two close scatterers clearly and stably for different noise levels. However, despite the index function computed by ``$G^\infty+$Partial Data" can take large values at the domain where the scatterers occupy, there are some distortions near the scatterers, which prevent us from accurately approximating the sizes and shapes of the scatterers. By using the probing function trained from the deep probing network, the computed index function can significantly achieve sharper and more accurate approximation for the locations and sizes of the scatterers. The performance of methods FFSM and FSSM with the regularization parameters $\{0.1^2,0.1^4\}$ seems to be not as good as that of the deep probing network, while with proper regularization parameter, they can also achieve some improvements over ``$G^\infty+$Partial Data".

\textbf{Example 2.2.}  In this example, the inhomogeneous scatterers are simulated by two rectangles. Three incidence plane waves with directions $d_1=(1,0),d_2=(-\frac{1}{2},\frac{\sqrt{3}}{2})$ and $d_3=(-\frac{1}{2},-\frac{\sqrt{3}}{2})$ are employed to generate the corresponding far-field data.  The reconstructions for this example with different noise levels are presented in Fig.\ref{fig: exam622}.

\begin{figure}[hpt]
\centering
	\begin{subfigure}{0.9\textwidth}
		\centering
		\includegraphics[width=\linewidth]{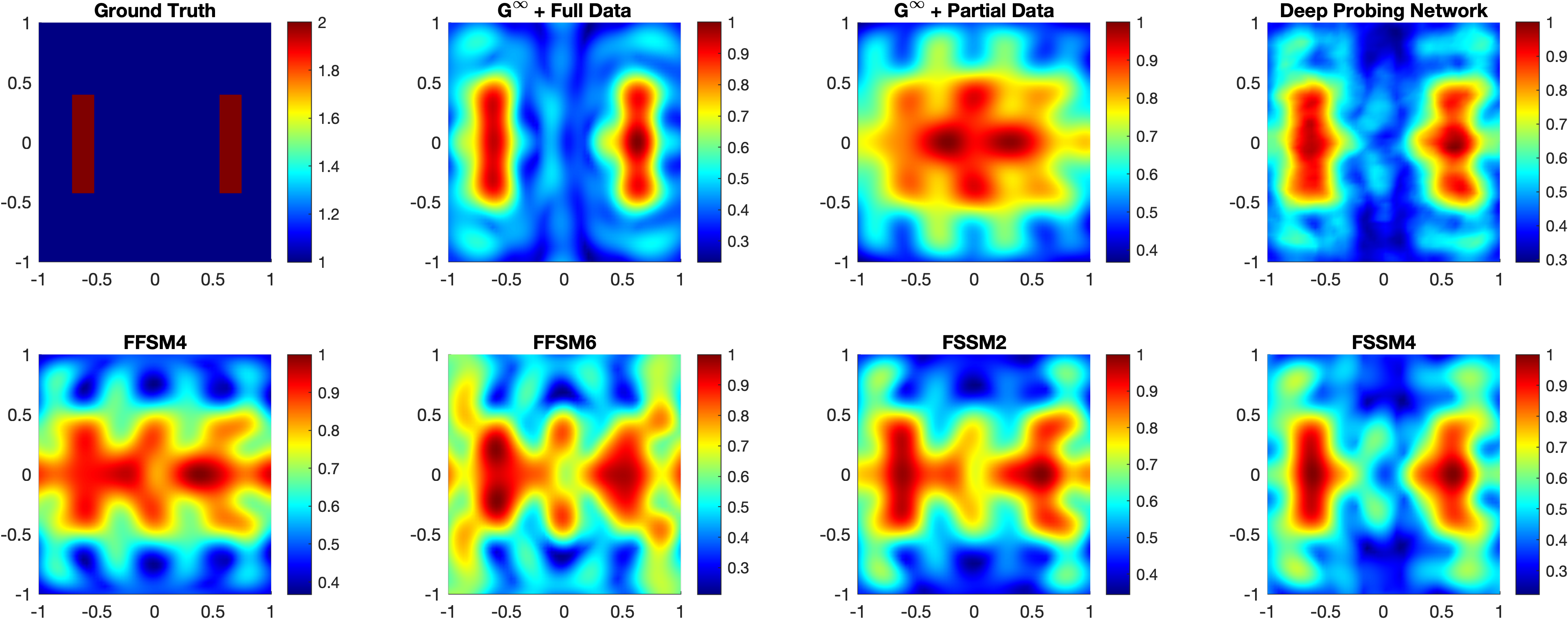}
		\caption{Reconstructions for \textbf{Example 2.2.} with noise level $\delta=1\%$}
		\label{fig: exam6221}
   \vspace{0.5cm}
	\end{subfigure}
	\begin{subfigure}{0.9\textwidth}
		\centering
		\includegraphics[width=\linewidth]{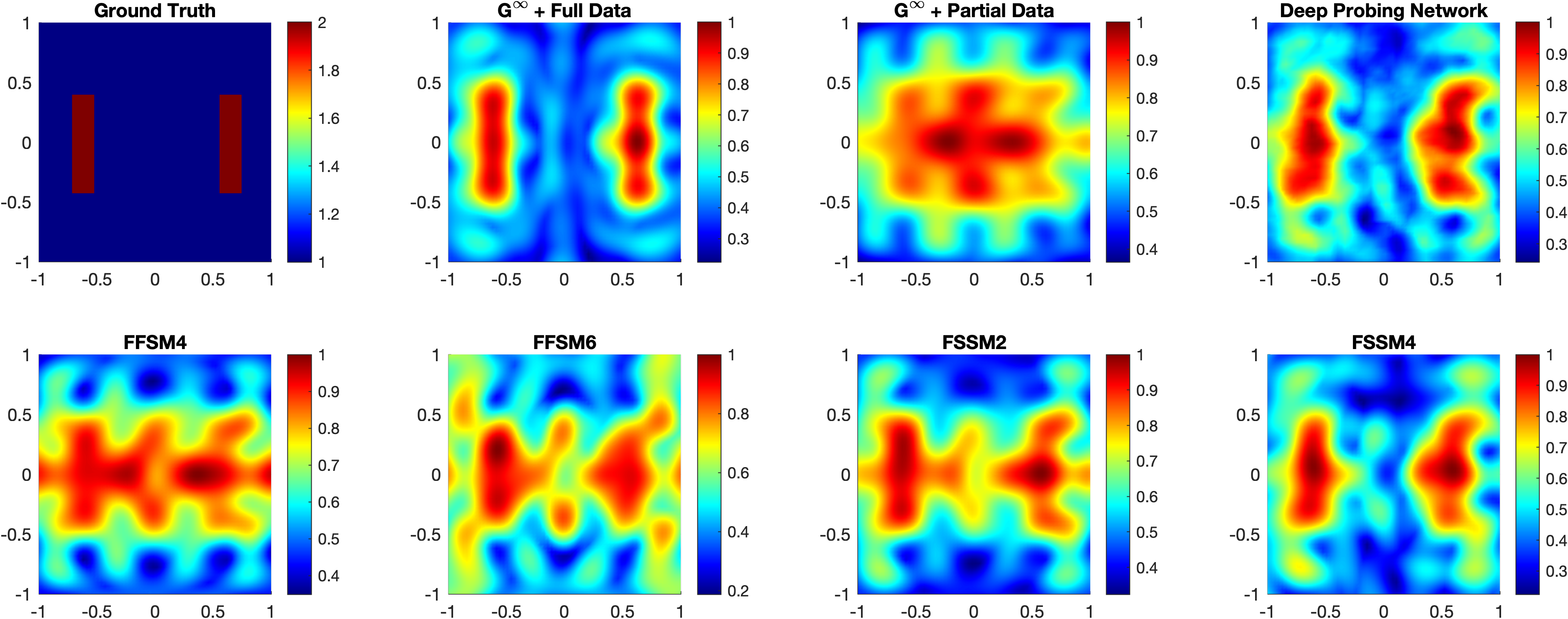}
		\caption{Reconstructions for \textbf{Example 2.2.} with noise level $\delta=5\%$}
		\label{fig: exam62210}
	\end{subfigure}
	\caption{Reconstructions for \textbf{Example 2.2.} with different noise levels.}
	\label{fig: exam622}
\end{figure}

As shown in Fig.\ref{fig: exam622}, the DSM with full-aperture data from only three incidences can very accurately and robustly approximate the scatterers' shapes, locations, and sizes. However, despite its robustness for high noise levels, the DSM with limited-aperture data cannot distinguish the two well-separated scatterers. With noise level $\delta=1\%$, the deep probing network can provide a quite satisfactory reconstruction for this example with limited-aperture data. An accurate result for this example can also be obtained by FSSM with regularization parameter $\sigma=0.1^4$. With regularization parameter $\sigma=0.1^6$, we can see that the method FFSM can also achieve better results than ``$G^\infty+$Partial Data", but they are still not very accurate compared to the deep probing network and ``FSSM4". With high noise level $\delta=5\%$,  we can see some effects of the high-level noise on the recovered shapes of the scatterers by the deep probing network and FSSM4, while the reconstructions are still reasonable and can distinct the two scatterers. Method FFSM with $\sigma=0.1^6$ can also provide better reconstruction than ``$G^\infty+$Partial Data" under noise level $\delta=5\%$. For large $\sigma$, methods FFSM and FSSM are very robust with the sacrifice on the resolution. Overall, with limited-aperture data, the deep probing network and FSSM with $\sigma=0.1^4$ achieve the best results for this example.

\section{Concluding remarks and future work}
\label{sec:conclu}
In this work, we have studied the DSMs for inverse scattering problems under highly challenging conditions, where the data are measured from limited-aperture receivers with only one or a few incident waves employed. The classical probing function is shown to be able to provide low-resolution approximation with limited-aperture data. To break this limitation, we propose a finite space framework with two specified methods and a deep learning scheme to construct effective probing functions for the DSMs with limited-aperture measurement. In addition, the partial measurement area can be very general. Several numerical experiments are presented to illustrate and compare different methods proposed in this paper. In addition, the proposed methods can be naturally extended to 3-dimensional space and potentially extended to the direct sampling methods for other inverse problems\cite{ito2013direct,chow2022direct,chow2014direct,chow2015direct,chow2021direct,chow2021directRadon} with limited data. 

In conjunction with our present work, several promising directions for future exploration exist. For the methods based on the finite space framework, in addition to the choices of the trail space and the testing space, another important research direction is to consider more suitable regularization schemes to solve the ill-posed linear system involved in this framework as well as to study the corresponding error analysis. In the numerical experiments of this paper, we solve \eqref{equa:FEMf2} with the same regularization parameter for all $z\in\Omega$, while considering the right side of this equation, it is reasonable to introduce a $z-$dependent parameter function $\sigma(z)$. The performance of the deep learning scheme can also be further improved by designing a more sophisticated network architecture or a more proper loss function and learning scheme. In addition, note that the learned probing function in the proposed method may need to be revised under specific noise levels. Thus, an interesting research direction is to learn a noise-dependent probing function with the form $G_\Gamma(z,\hat{x},\delta)$ where $\delta$ refers to the noise level.

\bibliography{bibfile}

\end{document}